\documentclass[11pt,reqno]{amsart}
\usepackage{a4wide}
\usepackage[
backend=bibtex,
style=numeric,
sorting=nyt,
maxbibnames=50,
giveninits=true,
maxalphanames=10
]{biblatex}
\usepackage{csquotes}
\renewbibmacro{in:}{%
	\ifentrytype{article}{}{}}
\DeclareFieldFormat[article]{volume}{\textbf{#1}}
\DeclareFieldFormat[article]{title}{\textit{#1}}
\DeclareFieldFormat[article]{journaltitle}{#1}
\DeclareFieldFormat[article]{number}{\textbf{#1}}
\renewbibmacro*{volume+number+eid}{%
	\printfield{volume}%
	\setunit*{\textbf{\addcolon}}%<---- was \setunit*{\adddot}%
	\printfield{number}\setunit{\addcomma\space}
	\printfield{eid}}

\usepackage{verbatim}
\usepackage[shortlabels]{enumitem}
\usepackage[T1]{fontenc}
\usepackage[unicode,hypertexnames=false,colorlinks=true,linkcolor=blue,citecolor=blue]{hyperref}
\usepackage{esint}
\usepackage{hyperref}
\usepackage[capitalise, nameinlink, noabbrev]{cleveref} %Leave it as the last package

\hypersetup{colorlinks,breaklinks,
	linkcolor=[rgb]{0,0,0},
	citecolor=[rgb]{0,0,0},
	urlcolor=[rgb]{0,0,0}}

\usepackage{xfrac,xcolor,color,graphicx}
\newcommand{\be}{\begin{equation}}
	\newcommand{\ee}{\end{equation}}
\theoremstyle{plain} %definition

\newtheorem{theorem}{Theorem}[section]
\newtheorem{proposition}[theorem]{Proposition}

\newtheorem{lemma}[theorem]{Lemma}
\newtheorem{definition}[theorem]{Definition}

\theoremstyle{remark}
\newtheorem{remark}[theorem]{Remark}

\newcommand{\res}{\mathop{\hbox{\vrule height 7pt width .5pt depth 0pt
			\vrule height .5pt width 6pt depth 0pt}}\nolimits}

\newcommand\ddfrac[2]{\frac{\displaystyle #1}{\displaystyle #2}}

\newcommand{\R}{\mathbb{R}}

\newcommand{\N}{\mathbb{N}}

\newcommand{\eps}{\varepsilon}
\newcommand{\F}{\mathcal{F}}
\newcommand{\mathbbmm}[1]{\text{\usefont{U}{bbm}{m}{n}#1}}
\newcommand{\ind}{\mathbbmm{1}}
\newcommand{\HH}{\mathcal{H}}

\newcommand{\norm}[2]{\Vert {#1} \Vert_{#2}}
\def\O{\Omega}
\DeclareMathOperator{\dive}{div}
\newcommand{\bea}{\begin{equation*}\begin{aligned}}
		\newcommand{\eea}{\end{aligned}\end{equation*}}

\makeatletter
\newcommand*{\wackyenum}[1]{%
	\expandafter\@wackyenum\csname c@#1\endcsname%
}

\newcommand*{\@wackyenum}[1]{%
	$\ifcase#1\or(H_\F)\or(H_f)\or(H_g)\or(H_p)%
	\else\@ctrerr\fi$%
}
\AddEnumerateCounter{\wackyenum}{\@wackyenum}{53.13}
\makeatother

\crefname{subsection}{subsection}{subsections}

\numberwithin{equation}{section}

\numberwithin{equation}{section}

\title{Smoothness and stability in the Alt-Phillips problem}

\author[M. Carducci]{Matteo Carducci}\thanks{}
\address {Matteo Carducci \newline \indent
	Classe di Scienze, Scuola Normale Superiore \newline \indent
	Piazza dei Cavalieri 7, 56126 Pisa - ITALY}
\email{\href{mailto:matteo.carducci@sns.it}{matteo.carducci@sns.it}}
\author[G. Tortone]{Giorgio Tortone}\thanks{}
\address {Giorgio Tortone \newline \indent
	Dipartimento di Matematica ``Giuseppe Peano'', Università di Torino \newline \indent
Via Carlo Alberto 10, 10123 Torino - ITALY}
\email{\href{mailto:giorgio.tortone@unito.it}{giorgio.tortone@unito.it}}

\begin{document}
	\begin{abstract} 
    We study the one-phase Alt-Phillips free boundary problem, focusing on the case of negative exponents $\gamma \in (-2,0)$. The goal of this paper is twofold.\\
    On the one hand, we prove smoothness of $C^{1,\alpha}$-regular free boundaries by reducing the problem to a class of degenerate quasilinear PDEs, for which we establish Schauder estimates. 
Such method provide a unified proof of the smoothness for general exponents.\\
On the other hand, by exploiting the higher regularity of solutions, we derive a new stability condition for the Alt-Phillips problem in the negative exponent regime, ruling out the existence of nontrivial axially symmetric stable cones in low dimensions. Finally, we provide a variational criterion for the stability of cones in the Alt-Phillips problem, which recovers the one for minimal surfaces in the singular limit as $\gamma \to -2$.
	\end{abstract}
	
	\maketitle
	\section{Introduction} 
    The one-phase Alt-Phillips free boundary problem concerns the study of non-negative minimizers of the functional 
\bea\label{e:functional}\mathcal{J}_\gamma(u):=\int_{B_1}(|\nabla u|^2 + u^\gamma \ind_{\{u>0\}})\,dx,\quad \mbox{where }\gamma\in(-2,2),\eea among functions with non-negative boundary datum on $\partial B_1$. 
The functional was first introduced in the 80s by Phillips in \cite{phillips}, and later further investigated by Alt and Phillips in \cite{Alt-Phillips}. More broadly, variational problems involving the potential $W(t) := t^\gamma \ind_{\{t>0\}}$ serve as a prototypical framework for studying general semilinear and quasilinear free boundary problems. When the potential lacks $C^{1,1}$-regularity (e.g.,~$\gamma \in (-2,2)$), it is natural to expect that minimizers may exhibit regions where the solution remains constant, thereby giving rise to a free boundary.

Beyond its original applications in modeling population dynamics \cite{GM}, and porous catalysts \cite{A}, the Alt-Phillips problem is especially notable for its role in ``interpolating'' between several classical and well-studied free boundary problems. Specifically: 
\begin{enumerate}[label=(\roman*)]
    \item $\gamma=1$, the obstacle problem;
    \item $\gamma =0$, the Alt-Caffarelli problem (i.e.,~the Bernoulli  problem);
    \item $\gamma \to -2$, the functional ``approximates'' the problem of minimal surfaces (see \cite{DeSilvaSavinNegativePower,DeSilvaSavinNegativePowerCompactness}). 
\end{enumerate}
Although the regularity theory for both the obstacle problem, the Alt-Caffarelli problem, and the minimal surfaces is fairly well understood (e.g.,~\cite{simonbook,CaffarelliSalsa:GeomApproachToFreeBoundary,psu12,maggi,Velichkov:RegularityOnePhaseFreeBd,xavixavi}), 
it is only in recent years that several research groups have developed a fine regularity theory for minimizers of the Alt-Phillips functional \cite{DeSilvaSavinPositivePower,DeSilvaSavinNegativePower,DeSilvaSavinNegativePowerUniformDensity,DeSilvaSavinNegativePowerCompactness,fk24,aks25,thomas,ros-restrepo,savinyu-altphillips-stable,savinyu-altphillips-concentration}.
\vspace{-0.35cm}\\

The focus of this work is the case of negative exponents $\gamma \in (-2,0)$, which has been introduced 
by De Silva and Savin in \cite{DeSilvaSavinNegativePower}. Precisely, as in the case of positive exponents, they show that minimizers exist by direct method and they are $\beta$-H\"{o}lder continuous, with 
\be\label{e:beta}
\beta := \frac{2}{2-\gamma}. 
\ee
This regularity is optimal, indeed the exponent $\beta$ is the scaling parameter of the Alt-Phillips functional for every exponent. Then, if we set $\O_u:=\{u>0\}$, they proved that the free boundary $\partial \O_u$ can be decomposed as the disjoint union 
$$
\partial \O_u \cap B_1 = \mathrm{Reg}(\partial\O_u)\cup \mathrm{Sing}(\partial\O_u),
$$
where the regular part $\mathrm{Reg}(\partial\O_u)$ is locally the graph of a $C^{1,\alpha}$ function and $\mathrm{Sing}(\partial\O_u)$ is a closed singular
set of Hausdorff dimension at most $d-d^*(\gamma)$, for some $d^*(\gamma)\geq 3$. In \cite{DeSilvaSavinNegativePower}, they observed one of the first remarkable differences between the cases of negative and positive exponents: indeed, the change of sign in the exponent makes the problem more degenerate and the free boundary condition is derived in terms of a second order expansion of the solution.\\
Furthermore, using a monotonicity formula and a dimension reduction argument, they showed
that the dimensional threshold $d^*(\gamma)$ is the first dimension in which
a minimizing cone for the Alt-Phillips problem exhibits singularities. Notice that, by the works of Caffarelli, Jerison and Kenig \cite{CaffarelliJerisonKenig04:NoSingularCones3D}, Jerison and Savin \cite{JerisonSavin15:NoSingularCones4D}, De Silva and Jerison \cite{DeSilvaJerison09:SingularConesIn7D}, we know that for the Alt-Caffarelli problem, i.e.,~$\gamma =0$, minimizing cones are flat in $\R^d$ for every dimension smaller than $d^*(0)\in \{5,6,7\}$ (we stress that the proof works also for stable cones).

Afterwards, in \cite{DeSilvaSavinNegativePowerCompactness} they improved the estimate of $d^*(\gamma)$ by exploiting a distinctive phenomenon that arises as $\gamma$ approaches the limiting values $-2$ and $0$. Indeed, by a fine compactness argument, they proved that the functionals $\Gamma$-converge to a Dirichlet-perimeter functional introduced by Athanasopoulos, Caffarelli, Kenig and Salsa in \cite{acks}. Therefore, by classical results concerning the dimension of the singular set for area-minimizing hypersurfaces (see \cite{simonbook,maggi}), they proved that asymptotically
$$
\text{if $\,\gamma \sim -2\,$ then $\,d^*(\gamma)\geq 8$},\qquad
\text{if $\,\gamma \sim 0\,$ then $\,d^*(\gamma)\geq 5$}.
$$

\subsection{Two equivalent formulations}\label{s:subs-auxiliary} Before presenting the main results, it is worth introducing an equivalent formulation of the Alt-Phillips problem, originally proposed in \cite[Section 4]{Alt-Phillips}. Precisely, for every exponent $\gamma \in (-2,2)$, given a minimizer $u$ of $\mathcal{J}_\gamma$ and letting $\beta$ be as in \eqref{e:beta}, we set $w:=\beta u^{1/\beta}$. Then $w$ is a minimizer of the functional 
\bea\label{e:functional-w}
\mathcal{E}_s(w):=\int_{B_1} w^s (|\nabla w|^2 +1) \ind_{\{w>0\}}
\,dx,\quad \mbox{where }s:=\beta\gamma \in (-1,+\infty).
\eea
 The auxiliary function $w$ has been exploited both for the case of positive 
 \cite{Alt-Phillips,DeSilvaSavinPositivePower,thomas}
 and negative exponents \cite{DeSilvaSavinNegativePower,DeSilvaSavinNegativePowerCompactness}. In this last contributions, the authors showed the convenience of this formulation for the study of regular points. Nevertheless, the variables $u$ and $w$ are interchangeable, since $\O_u\equiv \O_w$, and thus our analysis of the regularity of the free boundary can equivalently be formulated in terms of $w$. In the case of negative exponents $\gamma\in(-2,0)$, as initially highlighted by \cite{DeSilvaSavinNegativePower}, minimizers of $\mathcal{J}_\gamma$ are solutions of the Euler-Lagrange equations
\bea\label{e:EL-u}
\Delta u = \frac{\gamma}{2}u^{\gamma-1}\quad\mbox{in }\O_u\cap B_1,\qquad u^\gamma\left(\frac{|\nabla u|^2}{u^\gamma} - 1\right)=0\quad\mbox{on }\partial\O_u\cap B_1.
\eea
Equivalently, the free boundary condition can be rewritten as
\bea\label{e:visc.DeS.u}
\lim_{t \to 0^+} t^{2(\beta -1)}\left(\frac{u(x_0-t\nu_{x_0})}{t^\beta} - c_\beta \right) =0, \quad\mbox{where } c_\beta:=  \frac{1}{\beta^{\beta}} = \left(\frac{2-\gamma}{2}\right)^{\frac{2}{2-\gamma}},
\eea
where $x_0 \in \partial \O_u \cap B_1$ and $\nu_{x_0}$ is the outer normal to $\partial\O_u$ at $x_0$. We stress that $c_\beta$ is the constant associated to the one-dimensional solution $u_0(t) :=c_\beta (t^+)^\beta$.\\
On the other hand, a direct computation shows that the Alt–Phillips problem can be viewed as a degenerate one-phase type free boundary problem, where $w$ satisfies 
\bea\label{e:EL-w}
\Delta w = \frac{s}{2}\frac{1-|\nabla w|^2}{w}\quad\mbox{in }\O_w\cap B_1,\qquad
w^s(|\nabla w|^2 -1) = 0\quad\mbox{on }\partial\O_w\cap B_1.
\eea
Equivalently, the free boundary condition for $w$ is \be\label{e:visc.DeS.w}
\lim_{t \to 0^+} t^{s}\left(\frac{w(x_0-t\nu_{x_0})}{t} - 1 \right) =0 \quad\mbox{at }x_0\in \partial\O_w\cap B_1.
\ee
As noted in \cite{DeSilvaSavinNegativePower}, unlike in the case of positive exponents, the free boundary condition cannot be easily derived through integration by parts or inner variations. We refer to \cite[Proposition 4.4]{DeSilvaSavinNegativePower}, where such a condition is obtained for minimizers via a calibration argument. Ultimately, by proving an improvement-of-flatness lemma for viscosity solutions of the auxiliary formulation, they deduce $C^{1,\alpha}$-regularity for both the free boundary and the auxiliary function, in a neighborhood of regular points. We postpone the discussion of the natural free boundary condition for the Alt-Phillips problem to \cref{r:soddisfazione}.\vspace{-0.35cm}\\

The principle intention of this paper is to establish higher regularity of the free boundary at regular points and define the notion of stable singular cones for the Alt-Phillips functional in the negative regime, highlighting the connection with the stability inequalities for the Alt-Caffarelli functional \cite{CaffarelliJerisonKenig04:NoSingularCones3D} and for area-minimizing hypersurfaces \cite{Simons:minimal-varieties}.

\subsection{\texorpdfstring{$C^\infty$}{Cinfty} regularity of regular points}
Higher regularity of the free boundary in the Alt-Phillips problem has been recently established only in the range of positive exponents, with different methodologies. Precisely, at regular points, the smoothness was first proved for $\gamma \in (1,2)$ in \cite{fk24}, then extended to $\gamma \in (2/3,2)$ in \cite{aks25}, and more generally to all positive exponents $\gamma \in (0,2)$ in \cite{ros-restrepo} (see also \cite{dss25}).\\\vspace{-0.35cm} 

In the following theorem, we conclude the study of higher regularity for the Alt-Phillips problem, by proving smoothness of the free boundary at regular points in the case of negative exponents $\gamma\in(-2,0)$. 
Since our approach naturally extends to the full range of exponents $\gamma \in (-2,2)$, 
we provide an alternative and unified proof for the smoothness of the regular part of the free boundary.
\begin{theorem}\label{t:main-AP}
Let $\gamma \in (-2,2)$ and $u$ be a local minimizer of the Alt-Phillips functional $\mathcal{J}_\gamma$ in $B_1$. Assume that $x_0 \in\partial \O_u\cap B_1$ is a regular free boundary point. Then:
\begin{enumerate}
    \item[(i)] the free boundary $\partial\Omega_u$ is the graph of a smooth function in a neighborhood of $x_0$;\\\vspace{-0.35cm}
    \item[(ii)] $
w:= \beta u^{1/\beta}$ and $u/{\mathrm{dist}(x,\partial \O_u)^\beta}$
are smooth in a neighborhood of $x_0$ in $\overline{\Omega}_u$.
\end{enumerate}
\end{theorem}
Although \cref{t:main-AP} is stated in terms of local minimizers (see \cref{def:definition3.2}),  
the proof actually holds for regular solutions of the corresponding Euler-Lagrange equations, in the sense of \cref{def:definition2.1}. See \cref{prop:k-implies-k+1} and \cref{l:fbcondition} for the precise statements.\vspace{-0.35cm}\\

Our approach to prove \cref{t:main-AP} is to apply an hodograph transformation (see the seminal work of Kinderlehrer and Nirenberg \cite{KinderlehrerNirenberg1977:AnalyticFreeBd}) in a neighborhood of regular points, which reformulates the higher regularity problem 
as a Schauder-type estimates for degenerate quasilinear PDEs with Neumann boundary condition. More precisely, we consider 
the change of coordinates $\Phi(x',x_d):=(x',w(x',x_d))$
induced by the function $w$, and we show the existence of a $C^{1,\alpha}$-regular function $h:B_\delta\cap \{x_d\ge0\} \to \R$ satisfying
$$
w(x',h(x',x_d))=\beta u^{1/\beta}(x',h(x',x_d))= x_d,
$$
so that the free boundary 
is given by the graph of the trace of $h$ over $\{x_d=0\}$. 
If we assume that $w$ is a local minimizer of $\mathcal{E}_s$ (resp.~$u$ a local minimizer of $\mathcal{J}_\gamma$), then the hodograph transform $h$ is a local minimizer of 
$$
\mathcal{F}(\varphi):= \int_{B_\delta^+}x_d^{s}\, F(\nabla \varphi)\,dx,\quad \mbox{where } F(p):=\frac{|p|^2+1}{p_d}, 
$$
and the minimality is formulated in terms of outer variations. Therefore, the hodograph transform $h$ satisfies the following degenerate quasilinear PDE 
\bea\label{e:quasil-intro}
\mathrm{div}\left(x_d^s\, DF(\nabla h)\right)=0 \quad\mbox{in }B_\delta^+,\qquad \lim_{x_d\to 0^+}x_d^s\, DF(\nabla h)\cdot e_d = 0 \quad\mbox{on }B_\delta'.
\eea Here we denoted $B_\delta^+:=B_\delta\cap\{x_d>0\}$ and $B_\delta':=B_\delta\cap\{x_d=0\}$. 
Notice that:
\begin{itemize}
    \item[(i)] for $\gamma \in (-2, 2/3)$, i.e.,~$s \in (-1,1)$, the weight $x_d^s$ lies in the $A_2$-Muckenhoupt;
    \item[(ii)] for $\gamma \in [2/3, 2)$, i.e.,~$s\geq 1$, the weight is merely integrable.
\end{itemize} 
Finally, being $x'\mapsto h(x',0)$ a local parametrization of the free boundary, we deduce the claimed result from the validity of Schauder-type estimates for the function $h$ (see \cref{t:shauder-quasi} below). \\ \vspace{-0.35cm}

We emphasize that in \cite{ros-restrepo}, Restrepo and Ros-Oton adopt a different strategy, focusing on the PDE satisfied by the ratio of partial derivatives of a solution (see also \cite{DeSilvaSavin:higer-boundary-harnack-AC,DeSilvaSavin:higer-boundary-harnack-obstacle} in which this method was originally introduced in the context of the Alt-Caffarelli problem and the obstacle problem). 

\subsection{Degenerate quasilinear PDEs}
The regularity theory for degenerate divergence-form PDEs has been extensively developed over the past decades, both in the setting of $A_2$-Muckenhoupt weights \cite{kufneropic,fabeskenigserapioni} and in the so-called superdegenerate regime \cite{TTV}.

Our approach builds upon the analysis carried out by Terracini, Vita, and the second author in \cite{TTV}, where a Schauder-type theory is established for linear equations with integrable weights. However, to the best of our knowledge, a corresponding theory for quasilinear operators in the same setting is currently unknown. For this reason, our regularity result is formulated for general quasilinear PDEs, as we believe this result is of independent interest.
\begin{theorem}\label{t:shauder-quasi}
        Let $s>-1$, $k\in\N_{\ge2}$ and $\alpha\in(0,1)$. Let $v\in C^{1,\alpha}(\overline{B_1^+})$ be a solution of\medskip
\bea
    \mathrm{div}(x_d^s\,DF(\nabla v))=0\quad\text{in } B_1^+, \qquad
    \lim_{x_d\to 0^+}x_d^{s}\, D F(\nabla v)\cdot e_d =0\quad\text{on } B_1',
\eea for some $F:E \subset \R^d\to\R$ with $(\nabla v)(\overline{B_1^+})\subset E$, where $E\subset\R^d$ is a connected bounded open set. Assume that $F\in C^{k,\alpha}(E)$ is uniformly convex in $E$, i.e.,~there exist $0<\lambda\leq \Lambda<+\infty$ such that 
$$
\lambda\mathrm{Id}\le D^2 F(p)\le \Lambda \mathrm{Id}\quad\text{for every }p\in E.
$$
Then $v\in C^{k,\alpha}(\overline{B_{1/2}^+})$ and 
$$
\|v\|_{C^{k,\alpha}(\overline{B^+_{1/2}})}\le C, 
$$ 
for some constant $C>0$ depending only on $d, s, k, \alpha, \lambda, \Lambda, \|F\|_{C^{k,\alpha}(E)}, \|v\|_{C^{1,\alpha}(\overline{ B_1^+})}$.
    \end{theorem}
\subsection{Stable solutions of the Alt-Phillips problem for $\gamma \in (-2,0)$} 
In the Alt-Phillips problem, the classification of minimizing cones, i.e.,~homogeneous minimizers with isolated singularity, is a challenging open problem, at least in dimension $d\ge3$. Even for the Alt-Caffarelli problem ($\gamma=0$) the classification is only partially understood (see \cite{CaffarelliJerisonKenig04:NoSingularCones3D,DeSilvaJerison09:SingularConesIn7D,JerisonSavin15:NoSingularCones4D}). A key tool in addressing these questions is the concept of stability, which plays a central role in ruling out singularities by establishing the rigidity of singular cones in low dimensions (see \cite{CaffarelliJerisonKenig04:NoSingularCones3D}). In this direction, in \cite{thomas} Karakhanyan and Sanz-Perela recently computed the second variation for the Alt-Phillips problem with positive exponent.\\\vspace{-0.35cm}

The following is the main result concerning the stability condition under inner variations for minimizing cones in the case of negative exponents.
The smoothness of the regular part of the free boundary proved in \cref{t:main-AP} allows for second-order expansions of local minimizers, which are essential for computing the second variation.

\begin{theorem}\label{t:sing.min.cones}
    Let $\gamma \in (-2,0)$ and $u\in C^{0,\beta}(\R^d)$ be a $\beta$-homogeneous global  minimizer of $\mathcal{J}_\gamma$ in $\R^d$.
    Assume that $\partial\O_u$ is $C^{1,\alpha}$-regular outside the origin, then 
\be\label{e:stable1}
\int_{ \O_u} |\nabla u|^2 \left(|\nabla f|^2 - \mathcal{A}^2_u\,f^2\right)\,dx\geq 0,\quad \mbox{where }\mathcal{A}^2_u:= \frac{|\nabla^2 u|^2}{|\nabla u|^2} - \frac{|\nabla^2 u\nabla u|^2}{|\nabla u|^4},
\ee
for every $f \in C^\infty_c(\R^d\setminus\{0\})$. Equivalently, if $w:=\beta u^{1/\beta}\in C^{0,1}(\R^d)$ is $1$-homogeneous global minimizer of $\mathcal{E}_s$ in $\R^d$, and if $\partial\O_w$ is $C^{1,\alpha}$-regular outside the origin, then 
\be\label{e:stable2}
\int_{ \O_w} w^s|\nabla w|^2\left(|\nabla f|^2 - \mathcal{A}^2_w\, f^2\right)\,dx \geq 0,\quad \mbox{where }\mathcal{A}^2_w:= \frac{|\nabla^2 w|^2}{|\nabla w|^2} - \frac{|\nabla^2 w\nabla w|^2}{|\nabla w|^4},
\ee
for every $f \in C^\infty_c(\R^d\setminus\{0\})$. 
\end{theorem}

We stress that \cref{t:sing.min.cones} is stated under the assumptions of homogeneity and global minimality, 
but the proof actually applies in a more general setting. In fact, in \cref{def:definition3.1} we introduce the notion of stable solutions, and in \cref{prop:stability-cond} we prove a local stability condition without requiring homogeneity. Finally, in \cref{rem:holds-for-global}, we observe that \cref{t:sing.min.cones} holds even for global stable solutions, regular outside the origin, in the sense of \cref{def:definition3.2} and \cref{def:definition2.1}.\\
\vspace{-0.35cm}

For the computation of the stability condition, it is more convenient to rely on the variational formulation associated to $\mathcal{E}_s$, even though \eqref{e:stable1} and \eqref{e:stable2} are equivalent. Indeed, given $f \in C^\infty_c(\R^d \setminus \{0\})$, the condition \eqref{e:stable2} is deduced by computing the second derivative of the map 
\be\label{e:variations}
t \mapsto \mathcal{E}_s\left(w\circ (\mathrm{Id}+t \xi)^{-1}\right),\quad\mbox{where }\xi:=\frac{\nabla w}{|\nabla w|}f.
\ee
The choice of this variations is dictated by the strong degeneracy of the problem close to regular points. Indeed, while in the case of the Alt-Caffarelli problem ($\gamma = 0$) and for positive exponents $\gamma\in(0,2)$ it is natural to consider variations of the form $\xi := \frac{\nabla w}{|\nabla w|^2}f$ and integrate by parts on the free boundary to reveal the role of mean curvature, in the case of negative exponents the presence of the weight $w^s$, with $s \in (-1, 0)$, 
obstructs such approach and masks the contribution of the mean curvature.
%prevents one from highlighting the possible role of the mean curvature of the free boundary.
In this sense, the variations in \eqref{e:variations} are more close  used in the study of minimal surfaces, where the vector field is typically chosen of the form $\xi=\nu f$, with $\nu$ denoting the normal to the surface (see \cite{maggi}). We postpone the main differences between our stability condition and those in \cite{CaffarelliJerisonKenig04:NoSingularCones3D,thomas} for non-negative exponents to \cref{rem:differences}.

Notice that the quantity $\mathcal{A}_u^2$ in \cref{t:sing.min.cones} controls the second fundamental form of the level sets of the solution at a given point (the same applies to $\mathcal{A}_w$). Precisely, as pointed out by Sternberg and Zumbrun in \cite[Lemma 2.1]{sz98}, 
\be\label{szformula}
\mathcal{A}^2_u = |A|^2 + \frac{|\nabla_T |\nabla u||^2}{|\nabla u|^2}
\ee
where $\nabla_T$ denotes the tangential gradient along the level set of $u$ passing through $x\in \O_u$ and $|A|^2$ is the squared norm of the second fundamental form of the same level set.\\
\vspace{-0.35cm}

\subsection{Axially symmetric cones} Using the stability condition in \cref{t:sing.min.cones}, we provide an initial result concerning the existence of minimizers with an isolated singularity. Since in the regime of negative exponents the problem interpolates between the Alt-Caffarelli problem and the theory of minimal surfaces, it is natural to construct examples of singular free boundaries that are either axially symmetric (see \cite{DeSilvaJerison09:SingularConesIn7D} for $\gamma = 0$) or of Lawson’s type (see \cite{Simons:minimal-varieties, davini} for minimal surfaces). 

We emphasize that recently different groups of authors have focused on the study of singular point in the case of positive exponent. Indeed, in \cite{thomas} they ruled out the existence of singular cones that are axially symmetric in low dimensions, when the exponent $\gamma\in(0,2/3)$. On the other hand, in \cite{savinyu-altphillips-stable, savinyu-altphillips-concentration} the authors construct both radial and axially symmetric cones and show that these solutions are minimizers for $\gamma$ sufficiently close to $1$.

In the following theorem, we show that for negative exponents, minimizing axially symmetric cones for the Alt-Phillips functional are trivial in low dimensions.
\begin{theorem}\label{thm:axially}
Let $\gamma \in (-2,0)$ and $u \in C^{0,\beta}(\R^d)$ be a $\beta$-homogeneous global minimizer of $\mathcal{J}_\gamma$ in $\R^d$. Assume that $u$ is axially symmetric and 
$$
\text{either}\quad d\le6\quad \text{or}\quad d=7\ \text{and}\ \gamma<\frac{10-8\sqrt{5}}{11}\approx -0.7171.
$$
Then $u$ is one-dimensional.
\end{theorem}
Although \cref{thm:axially} is stated under the minimality assumption, the proof actually applies to the case of global stable solutions, regular outside the origin, in the sense of \cref{def:definition3.2} and \cref{def:definition2.1} (see \cref{proposition:d-2}). We also point out that the homogeneity assumption is not necessary (see \cref{rem:without-homo}). Finally, we observe that \cref{thm:axially} also holds for minimizers or stable solutions of the functional $\mathcal{E}_s$, under the corresponding assumptions.

\subsection{Asymptotic results as $\gamma \to -2$}
As already mentioned above, in \cite[Theorem 2.4]{DeSilvaSavinNegativePower}, the authors establish a $\Gamma$-convergence result of the (normalized) Alt-Phillips energies to the perimeter functional, as $\gamma \to -2$. This result provide a connection between minimizing free boundaries for the Alt-Phillips problem with area-minimizing hypersurfaces, in the same spirit of the theory of phase transition for the Allen-Cahn equation \cite{cahnhilliard,allencahn,modicamortola}.

It is natural to ask whether this connection between the problems persists at the level of stable solutions. For instance, the singular perturbation problem associated with stable solutions of the Allen-Cahn equation has been studied only recently  
(see \cite{Tonegawa,le1,TonWick,Hiesmayr,Gaspar}), where the limit has been characterized in terms of stable minimal hypersurfaces. On the other hand, we refer to \cite{KamburovWang,ChFeFiJo} for a recent literature concerning stable solutions of free boundary problems.\\

In this paper, we also investigate the limiting behavior of stable solutions of the Alt-Phillips problem as $\gamma \to -2$, by showing that under $C^{2,\alpha}$-regularity, uniform in $\gamma$ (see \eqref{e:melacito}), the limiting interface is a stable minimal hypersurface. 

First, in \cref{p:stability-criterion-eigen}, we provide a variational characterization of the stability condition of \cref{t:sing.min.cones} in terms of an eigenvalue problem involving a weighted Laplace-Beltrami operator on $\mathbb{S}^{d-1}$. Precisely, given a $\beta$-homogeneous solution $u$ and $\Sigma_u:=\Omega_u\cap \mathbb{S}^{d-1}$, we set 
 $$\lambda_s(\Sigma_u):= \min_{\substack{\varphi \in C^\infty(\mathbb{S}^{d-1})\\ \varphi\not\equiv0}}\ddfrac{\int_{\Sigma_u}  |\nabla u|^2\big(|\nabla_S \varphi|^2 - \mathcal{A}_u^2 \varphi^2\big)\,d\HH^{d-1}}{\int_{\Sigma_u} |\nabla u|^2 \varphi^2 \,d\HH^{d-1}}\,,$$
 where $\nabla_S$ is the tangential gradient on $\mathbb{S}^{d-1}$. Then, we show that the stability condition is equivalent to require that $$\lambda_s(\Sigma_u)\geq -\left(\frac{d+s-2}{2}\right)^2.$$
Finally, we address the asymptotic behavior of this criterion as $\gamma \to -2$, recovering the analogue result for stable minimal surfaces, thus obtaining that the limit interface is a stable minimal surface. For a discussion on the general limit and the role of the associated quadratic forms of the stability condition in \cref{t:sing.min.cones}, we refer to \cref{rem:quadratic}.\\
A valuable direction for future research, though beyond the scope of the present work, would be to expand on this analysis and extend the asymptotic result in \cite{Tonegawa,TonWick} to the case of stable solutions for the Alt-Phillips problem with negative exponents.

\subsection{Structure of the paper} 
In \cref{s:smooth}, we apply the hodograph map to regular solutions and we rewrite the problem in terms of a degenerate quasilinear elliptic PDE. Then we establish corresponding Schauder estimates in \cref{t:shauder-quasi}, which implies the $C^\infty$ regularity of the free boundary in \cref{t:main-AP}. In \cref{s:variations}, we compute both the first and second variations of the functional $\mathcal{J}_\gamma$, leading to the stability condition in \cref{t:sing.min.cones}. In \cref{sec:axially}, we use the stability condition to prove that axially symmetric cones are trivial in low dimensions, i.e.,~\cref{thm:axially}.
Finally, in \cref{sec:asymptotic} we prove the variational characterization of the stability condition and we discuss the asymptotic of stable cones as $\gamma\to-2$.

\section{\texorpdfstring{$C^\infty$}{Cinfty} regularity of the free boundary}\label{s:smooth} 
In this section, we show that the free boundary conditions are satisfied in a pointwise sense at regular points and then we rephrase the problem by an hodograph transformation. Therefore, the higher regularity of the free boundary follows by Schauder-type estimates for degenerate quasilinear PDEs with Neumann boundary conditions.

For the sake of clarity, we start by introducing the notion of regular solutions. For simplicity, the definition is given in terms of $w$, though an equivalent one holds for $u$.

\begin{definition}[Regular solutions]\label{def:definition2.1}
    Let $\gamma \in (-2,2)$ and $s:=\beta \gamma>-1$. We say that $w$ is a regular solution in an open set $D$ (unless otherwise specified, we take $D:=B_1$)
of
\be\label{e:hodo-pre-transform}
\Delta w = \frac{s}{2}\frac{1-|\nabla w|^2}{w}\quad\mbox{in }\O_w\cap D, \qquad 
\begin{cases}
|\nabla w|^2=1 &\mbox{on }\partial \O_w \cap D, \text{ if $s\geq 0$}\\
w^{s}(|\nabla w|^2-1)=0 &\mbox{on }\partial \O_w \cap D, \text{ if $s < 0$}
\end{cases}
\ee
if, for some $\alpha>-s$,
$$
\text{$w\in C^{1,\alpha}(\overline\Omega_w\cap D)$ and the free boundary $\partial\Omega_w$ is $C^{1,\alpha}$-regular in $D$} 
$$
and \eqref{e:hodo-pre-transform} is satisfied in a pointwise sense. Precisely, for $s\in (-1,0)$
\bea\label{e:condizione-puntuale}
\lim_{t \to 0^+} w^s(x_0-t\nu_{x_0})\Big(|\nabla w(x_0-t\nu_{x_0})|^2-1\Big)=0,\quad\mbox{for every }x_0 \in \partial \O_w \cap D,
\eea
where $\nu_{x_0}$ is the outer normal to $\partial\O_w$ at $x_0$.
\end{definition}
The following is the main regularity result, which ultimately implies \cref{t:main-AP}.
\begin{proposition}\label{prop:k-implies-k+1}
    Let $\gamma \in (-2,2)$ and $w$ be a regular solution in the sense of \cref{def:definition2.1}. If $0\in\partial\Omega_w$, then, there exists $r>0$ such that
$$
\text{$w\in C^{\infty}(\overline\Omega_w\cap B_r)$ and the free boundary $\partial\Omega_w$ is smooth in $B_r$.} 
$$
    \end{proposition}

Before addressing the problem of higher regularity, we first point out that local minimizers of the Alt-Phillips functional are regular solution, close to regular free boundary points. Such observation was essential contained in the $\eps$-regularity theorem proved in \cite{DeSilvaSavinPositivePower,DeSilvaSavinNegativePower}.

\begin{lemma}\label{l:fbcondition}
      Let $\gamma \in (-2,2)$ and $w$ be a local minimizer to $\mathcal{E}_s$, with $s:=\beta \gamma$. If $x_0\in\partial\Omega_w\cap B_1$ is a regular free boundary point, then $w$ is a regular solution in a ball $B_r(x_0)$, in the sense of \cref{def:definition2.1}, for some $r>0$.
\end{lemma}
\begin{proof} 
It is not restrictive to assume that $x_0=0$. Then, up to rescaling the problem, by \cite[Theorem 2.3]{DeSilvaSavinNegativePower} there exists $\delta \in (0,1)$ such that $ w\in C^{1,\delta}(\overline \Omega_w\cap B_r)$ and $\partial \O_w$ is $C^{1,\delta}$-regular in $B_r$, for some $r>0$. 
Since in the case of positive exponents $\gamma \in [0,2)$ the conclusion is immediate, we address the case $\gamma \in (-2,0)$. Notice that, for negative $\gamma$, the result follows once we show that 
$$
\text{$w\in C^{1,\alpha}(\overline\Omega_w\cap B_r)$ and the free boundary $\partial\Omega_w$ is $C^{1,\alpha}$-regular in $B_r$}, 
$$
for some $\alpha>-s$. In fact, being $|\nabla w|=1$ on $\partial \O_w\cap B_r$, it implies that% the result will naturally follow by the improved H\"{o}lder continuity of the gradient
$$
\lim_{t\to0^+}\left|w^s(x_0-t\nu_{x_0})\Big(|\nabla w(x_0-t\nu_{x_0})|^2-1\Big)\right| \leq C \lim_{t \to 0^+}t^{s+\alpha}=0.
$$
As already observed in the proof of \cite[Proposition 7.2]{DeSilvaSavinNegativePower}, the linearization procedure allows to rewrite the regularity problem for $\eps$-flat free boundaries in terms of Schauder-type estimates for the degenerate linear PDE
$$
\mathrm{div}\left(x_d^s\, \nabla v\right)=0 \quad\text{in }B_1^+,\qquad \lim_{x_d\to 0^+}x_d^s\, \nabla v\cdot e_d = 0 \quad\text{on }B_1'.
$$
Since solutions of such linearized system are indeed smooth up to the boundary (see \cite[Theorem 1.1]{STV}), we can show that for every $\alpha \in (0,1)$ there exists a dimensional radius $\rho>0$, small enough, such that 
$$\left(x\cdot\nu-\rho^{1+\alpha}\eps\right)_+\le w(x)\le \left(x\cdot \nu+\rho^{1+\alpha}\eps\right)_+\quad\text{for every } x\in B_\rho.$$
Then, fixed $\alpha>-s$, the result follows by a standard iteration argument (see for instance \cite{cv24}, where such strategy is exploited for obtaining $C^{2,\alpha}$-estimates in a one-phase type problem).
\end{proof}

   \subsection{The hodograph transform}\label{s:subs.hodograph} In this section we write the hodograph transformation of a solution of the Alt-Phillips problem (see also \cite{fk24,aks25} in which this map was already exploited in the case of positive exponent $\gamma \in (2/3,2)$).

Given $x \in \{x_d\geq 0\}$ it is convenient to write $x=(x',x_d)\in \R^{d-1}\times \R$. Without loss of generality, we can assume that $w$ is a regular solution in $B_1$, in the sense of \cref{def:definition2.1}, and that $0 \in \partial\O_w$ and $\nabla w(0)= e_d$.
Therefore, there exists $\rho>0$ sufficiently small, such that the map   
$$
\Phi:\overline\Omega_w\cap B_\rho\to \R^d\cap\{y_d\ge0\},\quad \Phi(x',x_d):=(x',w(x',x_d)),\quad 
$$
is bijective from $\overline\Omega_w\cap B_\rho$ onto an open neighborhood of the origin in the upper half-space $\{y_d\ge0\}$. Clearly, it is not restrictive to replace $\Phi(\overline\Omega_w\cap B_\rho)$ with $\{y_d\ge0\}\cap B_\delta$, for some $\delta>0$. %For the sake of simplicity, we use the following notations 
%$$
%B_r^+ := B_r\cap \{x_d>0\},\qquad B_r':= B_r\cap \{x_d=0\}.
%$$
Thus, $\Phi$ maps the free boundary $\partial\Omega_w\cap B_\rho$ into $B'_\delta$. On the other hand, the inverse function 
\be\label{hodographdef}\Phi^{-1}:\overline{B_\delta^+}\to \overline\Omega_w\cap B_\rho,\quad \Phi^{-1}(y',y_d):=(y',h(y',y_d)),\quad 
\ee 
is well-defined. 
Being $w\in C^{1,\alpha}(\overline\Omega_w\cap B_\rho)\cap C^\infty_{\mathrm{loc}}(\Omega_w\cap B_\rho)$, we get that $h: 
\overline{B_\delta^+} \to \R$ is $C^{1,\alpha}$-regular in $\overline{B_\delta^+}$ and $C^\infty$-regular in $B_\delta^+$.
Throughout the paper, we refer to such $h$ as the hodograph transform of $w$. 
By differentiating the identity
$$
h(x',w(x',x_d))=x_d, \quad\mbox{for every }x=(x',x_d) \in \overline\O_w \cap B_\delta, 
$$
we can derive the system satisfied by the hodograph transform $h$. For the sake of simplicity, in the computations that follow, the derivatives of $w$ are evaluated at $x\in \Omega_w\cap B_\delta$, and those of $h$ at $y:=(x',w(x',x_d))\in B_\delta^+$. Given $i,j=1,\dots,d-1$, we get
\begin{align}\label{e:computations}
\begin{aligned}
 & w_d =\frac{1}{h_d},\qquad w_i=-\frac{h_i}{h_d},\\[0.5em]
     &w_{ij}= -\frac{h_{ij}}{h_d}
     +\frac{ h_j h_{id}}{h_d^2}+\frac{h_{jd}h_i}{h_d^2}
     -\frac{h_{dd} h_i h_j}{h_d^3},\qquad w_{id}=-\frac{h_{id}}{h_d^2}+ \frac{h_{dd}\, h_i}{h_d^3},\qquad w_{dd}=- \frac{h_{dd}}{h_d^3}.
\end{aligned}
\end{align}
Although this is a classical fact (see e.g.,~\cite[Remark 3.1]{depsv}), we emphasize that the hodograph transform $h$ contains all the information of $\partial \O_w$, indeed $h$ is a regular function that coincides on $B_\rho'$ with the parametrization of the free boundary $\partial \O_w$. Thus, the derivatives of the graph of $\partial\Omega_w$ coincide with the partial derivatives $h_i$, for $i = 1, \dots, d-1$.

More precisely, if $\partial \O_w \cap B_\rho$ is the graph of the function $g : B_\rho'\to \R$ in the $e_d$-direction, then 
$$
w(x',g(x'))=0\quad \mbox{for every }(x',0)\in B_{\rho}'.
$$
Then, by differentiating the identity above, we obtain for $i=1,\dots,d-1$
\be\label{e:deriv.graph}
g_i(x')=-\frac{w_i(x',g(x'))}{w_d(x',g(x'))}= h_i(x',w(x',g(x'))) = h_i(x',0)\quad\mbox{for every }(x',0) \in B_{\rho}',
\ee
where the second equality follows from \eqref{e:computations} and the last from the definition of $g$. Since $e_d$ is the interior normal to $\partial \O_w$ at the origin, we have $\nabla h(0)=e_d$.
By exploiting \eqref{e:computations}, we have 
$$
\frac{s}{2}\frac{1-|\nabla w|^2}{w}= \frac{s}{2}\frac{1}{y_d}\left(1-\frac{1+\sum_{i=1}^{d-1} h_i^2}{h_d^2}\right).
$$
Hence, the interior condition \eqref{e:hodo-pre-transform} become
\be\label{eq:v-system-1} - \frac{h_{dd}}{h_{dd}^3}\left(1 + \sum_{i= 1}^{d-1} h_i^2\right)
+ 2\sum_{i= 1}^{d-1} \frac{h_{i}h_{id}}{h_d^2} -  \sum_{i=1}^{d-1}\frac{h_{ii}}{h_d} =\frac{s}{2y_d}\left(1-\frac{1+\sum_{i=1}^{d-1}h_i^2}{h_d^2}\right)\quad\text{in }B_\delta^+.
\ee 
Similarly, the free boundary condition in \eqref{e:hodo-pre-transform} can be rewritten as 
\be\label{eq:v-system-2}
\lim_{y_d \to 0^+} y_d^s\left(1-\frac{1+\sum_{i=1}^{d-1}h_i^2}{h_d^2}\right)=0\quad\text{on }B_\delta'.
\ee 

We now collect all the computations and we show that the hodograph transform satisfies a degenerate quasilinear elliptic PDEs with a Neumann boundary condition. For the sake of completeness, we refer to \cite{fabeskenigserapioni,kufneropic,TTV} for a complete discussion on weighted Sobolev spaces and the correct notion of weak solutions.
\begin{lemma}\label{l:hsolves} 
Let $w$ be a regular solution, in the sense of \cref{def:definition2.1}. Then its hodograph transform $h$, defined in \eqref{hodographdef}, is a solution of
\be\label{eq:quasi-v}
\mathrm{div}\left(x_d^s\, DF(\nabla h)\right)=0 \quad\text{in }B_\delta^+,\qquad \lim_{x_d\to 0^+}x_d^s\, DF(\nabla h)\cdot e_d = 0 \quad\text{on }B_\delta',
\ee
with \be\label{e:nostraF} F(p):=\frac{|p|^2 + 1 }{p_d}.
\ee
Moreover, if $w$ is a local minimizer of $\mathcal{E}_s$, then the hodograph transform $h$ minimizes  
$$\mathcal{F}(\varphi):= \int_{B_\delta^+}x_d^{s}\, F(\nabla \varphi)\,dx,$$
among competitors in $H^1(B_\delta^+,x_d^s\,dx)$ with same trace on $(\partial B_\delta)^+$. 

\end{lemma} 
\begin{proof}
We divide the proof into two cases. \\\vspace{-0.3cm}

\noindent \textit{Case 1: regular solutions of \eqref{e:hodo-pre-transform}.} First, by direct computation, we have
$$
D F(p) = \frac{2}{p_d}p -\frac{(1+|p|^2)}{p_d^2}e_d,\qquad 
D^2 F(p)=\frac{2}{p_d}\mathrm{Id}-2\frac{e_d\otimes p + p\otimes e_d}{p_d^2}+\frac{2(1+|p|^2)}{p_d^3}e_d\otimes e_d.
$$
Then, since
\be\label{e:utile}
\begin{aligned}
&DF(\nabla h)\cdot e_i = 2\frac{h_i}{h_d}\quad\mbox{for every }i=1,\dots,d-1,\\
&DF(\nabla h)\cdot e_d  = 2-\frac{1+|\nabla h|^2}{h_d^2}=1-\frac{1+\sum_{i=1}^{d-1}h_i^2}{h_d^2}
\end{aligned}
\ee
we get that \begingroup\allowdisplaybreaks\begin{align*} &y_d^{-s}\mathrm{div}\left( y_d^{s} DF(\nabla h)\right)=\sum_{i=1}^{d}\partial_i(DF(\nabla h)\cdot e_i)+s\frac{DF(\nabla h)\cdot e_d}{y_d}\\
&\qquad=2\sum_{i=1}^{d-1}\left(\frac{h_{ii}}{h_d}-\frac{h_{id}h_{i}}{h_d^2}\right)+\partial_d\left(-\frac{1+\sum_{i=1}^{d-1}h_i^2}{h_d^2}\right)+\frac s{y_d}\left(1-\frac{1+\sum_{i=1}^{d-1}h_i^2}{h_d^2}\right)\\
&\qquad=
    2\sum_{i=1}^{d-1}\left(\frac{h_{ii}}{h_d}-\frac{h_{id}h_{i}}{h_d^2}\right)+2\frac{h_{dd}}{h_d^3}-2\left(\sum_{i=1}^{d-1} \frac{h_i h_{id}}{h_d^2}-\frac{h_{dd}h_i^2}{h_d^3}\right)+\frac s{y_d}\left(1-\frac{1+\sum_{i=1}^{d-1}h_i^2}{h_d^2}\right)\\
&\qquad=2\sum_{i=1}^{d-1}\left(\frac{h_{ii}}{h_d}-2\frac{h_{id}h_{i}}{h_d^2}\right)+2\frac{h_{dd}}{h_d^3}\left(1+\sum_{i=1}^{d-1}h_i^2\right)+\frac s{y_d}\left(1-\frac{1+\sum_{i=1}^{d-1}h_i^2}{h_d^2}\right).
    \end{align*}\endgroup 
    Finally, the weighted quasilinear PDE follows by combining the last line with \eqref{eq:v-system-1} in $B_\delta^+$. Similarly, the boundary condition follows by comparing \eqref{eq:v-system-2} with the second line in \eqref{e:utile}.
\\\vspace{-0.3cm}

\noindent \textit{Case 2: local minimizers of $\mathcal{E}_s$.}
First, by \eqref{e:computations}, we have
$$
w^s(x)|\nabla w(x)|^2=  y_d^{s}\,\frac{|\nabla_{x'} h(y)|^2 + 1 }{h_d^2(y)}\quad\mbox{and }\quad\mathrm{det}\nabla \Phi(x) = w_d(x) = \frac{1}{h_d (y)} 
$$
where in both the equality we set $y:=\Phi(x)$. By applying the change of coordinate $y:=\Phi(x)$ to the functional $\mathcal{E}_s$, we get
$$
\int_{B_\rho \cap \O_w}w^s(|\nabla w|^2 + 1)\, dx = \int_{B_\delta^+} y_d^{s}\left(\frac{|\nabla_{x'} h|^2 + 1 }{h^2_d} + 1\right) h_d\,dy
$$
where the last factor comes from the first identity in \eqref{e:computations}. Therefore,
$$
\int_{B_\rho \cap \O_w}w^s(|\nabla w|^2 +1)\, dx = \int_{B_\delta^+} y_d^{s}\,F(\nabla h)\,dy
$$
and the Euler-Lagrange equations coincides with \eqref{eq:quasi-v}.
\end{proof}

\begin{remark}[Degenerate quasilinear problems and convexity]
In the context of degenerate quasilinear operators, it is natural to say that 
\eqref{l:hsolves} is uniformly elliptic if the function $F$ is strictly convex (see \cref{t:shauder-quasi}).
Precisely, when $F$ is given as in \eqref{e:nostraF}, and thus is not convex, we localize the problem so that \eqref{l:hsolves} is uniformly elliptic in a possibly smaller neighborhood of the origin (see the proof of \cref{prop:k-implies-k+1}).
\end{remark}
\begin{remark}[The free boundary condition]\label{r:soddisfazione} 
    In \cite{DeSilvaSavinPositivePower,DeSilvaSavinNegativePower}, the authors derived the free boundary condition associated to the Alt-Phillips problem, both for positive and negative exponents. Precisely, the function $w:=\beta u^{1/\beta}$ satisfies 
    $$
\begin{cases}
|\nabla w|^2=1 &\quad\mbox{on }\partial \O_w \cap B_1, \text{ if $\gamma \geq 0$}\\
w^{s}(|\nabla w|^2-1)=0 &\quad\mbox{on }\partial \O_w \cap B_1, \text{ if $\gamma < 0$}
\end{cases}
$$
in a viscosity sense.
    The understanding of the free boundary condition is fundamental for the proof of an improvement-of-flatness lemma by a linearization argument. Nevertheless, the boundary condition associated to the hodograph transform of $w$ (see the Neumann condition in \eqref{eq:quasi-v}) and the linearized problem arising from the linearization in \cite[Section 7]{DeSilvaSavinNegativePower} suggest that is natural to consider viscosity solution of
$$
w^{s}(|\nabla w|^2-1)=0 \quad\mbox{on }\partial \O_w \cap B_1,
$$
independently on the sign of the exponent $\gamma \in (-2,2)$. 
Moreover, although for positive powers this condition is a priori weaker, the analysis in \cite{TTV} suggests that a posteriori the two conditions are indeed equivalent.
\end{remark}

\subsection{Schauder estimates for degenerate quasilinear PDEs} 
In this section, we present a first result in the setting of Schauder estimates for degenerate quasilinear PDEs. 

Before stating the main Schauder-type estimates, we recall the following lemma from \cite[Lemma 2.4, Remark 2.5]{TTV}.
\begin{lemma}\label{lemma-ttv-utile}
        Let $s>-1$, $k\in \N$ and $\alpha\in (0,1)$. Let $f\in C^{k,\alpha}(\overline {B_1^+})$, and $$\varphi(x):=\frac{1}{x_d^s}\int_0^{x_d}t^sf(x',t)\,dt.$$ Then $\varphi_d\in C^{k,\alpha}(\overline {B_1^+})$ and the following estimate holds
        $$\|\varphi_d\|_{C^{k,\alpha}(\overline {B_{1/2}^+})}\le C\|f\|_{C^{k,\alpha}(\overline {B_{1/2}^+})},$$ for some $C>0$ depending only on $d$, $s$, $k$, $\alpha$.
    \end{lemma}
Ultimately, \cref{t:shauder-quasi} is a consequence of the following proposition. 
    \begin{proposition}\label{prop:quasi}
        Let $s>-1$, $k\in\N_{\ge1}, \alpha\in(0,1)$. Let $v\in C^{k,\alpha}(\overline{B_1^+})$ be a solution of\medskip 
\be\label{eq:problem-quasi}
    \mathrm{div}(x_d^s\,DF(\nabla v))=0\quad\text{in } B_1^+, \qquad
    \lim_{x_d\to 0^+}x_d^{s} D F(\nabla v)\cdot e_d =0\quad\text{on } B_1',
\ee for some $F: E\subset \R^d\to\R$ with $(\nabla v)(\overline{B_1^+})\subset E$, where $E\subset \R^d$ is a connected bounded open set. Assume that $F\in C^{k+1,\alpha}(E)$ is uniformly convex in $E$, i.e.,~there exist $0<\lambda\leq \Lambda<+\infty$ such that 
$$
\lambda\mathrm{Id}\le D^2 F(p)\le \Lambda \mathrm{Id}\quad\text{for every }p\in E.
$$
Then $v\in C^{k+1,\alpha}(\overline{B_{1/2}^+})$ and 
$$
\|v\|_{C^{k+1,\alpha}(\overline{B^+_{1/2}})}\le C, 
$$ 
for some constant $C>0$ depending only on $d, s, k, \alpha, \lambda, \Lambda, \|F\|_{C^{k+1,\alpha}(E)}, \|v\|_{C^{k,\alpha}(\overline{ B_1^+})}$.
    \end{proposition}
\begin{proof}
 The proof builds on the argument developed in the linear case in \cite{TTV}. As already noted in \cite{TTV}, the bootstrap argument must account for the differing behavior of the operator along the tangential directions $e_i$ for $i = 1, \dots, d-1$, and the vertical direction $e_d$. We mention that all the positive constants $C$ in the proof depend only on $d$, $s$, $k$, $\alpha$, $\lambda$, $\Lambda$, $\|F\|_{C^{k+1,\alpha}(E)}, \|v\|_{C^{k,\alpha}(\overline{ B_1^+})}$.
\vspace{-0,3cm}\\
 
 First, let $A(p):=D^2 F(p): E \to \R^{d\times d}$ be a matrix-valued function, uniformly elliptic in $E$, then we observe that for every $i=1,\ldots,d-1$, the tangential partial derivative $v_i$ solves the following linear degenerate elliptic problem
 \vspace{0.1cm} 
\be\label{eq:problem-quasi-tangential}
\mathrm{div}(x_d^s\, A(\nabla v)\nabla v_i )=0\quad\text{in } B_1^+,\qquad
   \lim_{x_d\to 0^+} x_d^{s}\, A(\nabla v)\nabla v_i\cdot e_d =0\quad\text{on } B_1'.
\ee 
Since the operator is differentiable in the tangential directions, \eqref{eq:problem-quasi-tangential} can be obtained using the difference quotient method. Then, being $A(\nabla v(\cdot))\in C^{k-1,\alpha}(\overline{B_1^+}),$ by \cite{TTV}, we deduce that $v_i\in C^{k,\alpha}(\overline{B_{1/2}^+})$ and \be\label{eq:derivativeiv}\|v_i\|_{C^{k,\alpha}(\overline{B^+_{1/2}})}\le C
\ee 
for some constant $C>0$.
\\\vspace{-0.35cm}

Regarding the normal derivative, set $\varphi:= DF(\nabla v)\cdot e_d$. Then, by the previous part of the proof, since $DF\in C^{k,\alpha}(E;\R^{d\times d})$ and $v_i\in C^{k,\alpha}(\overline{B_{1/2}^+})$ for every $i=1,\ldots,d-1$, we get \bea\label{eq:derivativei}
\varphi_i\in C^{k-1,\alpha}(\overline{B_{1/2}^+})\quad\text{for every}\quad i=1,\ldots,d-1,
\eea 
and, by \eqref{eq:derivativeiv}, \be\label{derivativevarphii}\|\varphi_i\|_{C^{k-1,\alpha}(\overline{ B_{1/2}^+})}\le C.\ee Then, let us prove that also the normal derivative $\varphi_d$ is $C^{k-1,\alpha}$-regular. First, we notice that the interior condition in \eqref{eq:problem-quasi} can be rewritten as $$\varphi_d+\frac{s}{x_d}\varphi=-\sum_{i=1}^{d-1}\partial_i(DF(\nabla v)\cdot e_i)=:f\qquad\mbox{in }B_1^+.$$ 
Notice, that $f\in C^{k-1,\alpha}(\overline{ B_{1/2}^+})$ since both $DF$ and $v_i$ are $C^{k,\alpha}$-regular, for $i=1,\ldots,d-1$. Now, fixed $(x',0)\in B_{1/2}'$ the above identity for $\varphi$ implies that
$$
x_d^{-s}\,\partial_d(x_d^s\, \varphi(x',x_d))= f(x',x_d)\quad\mbox{for }x_d \in (1/2-|x'|^2)^{1/2}. 
$$
Moreover, by solving the ODE with boundary condition $\lim_{x_d\to 0^+}x_d^s \, \varphi(x',x_d)=0$, we get 
$$\varphi(x',x_d)=\frac{1}{x_d^s}\int_0^{x_d} t^s f(x',t)\,dt\quad\mbox{in }B_{1/2}^+.
$$
By \cref{lemma-ttv-utile}, we have that $$\varphi_d\in C^{k-1,\alpha}(\overline{ B_{1/2}^+})$$ and 
\be\label{derivativevarphid}
\|\varphi_d\|_{C^{k-1,\alpha}(\overline{ B_{1/2}^+})}\le C\|f\|_{C^{k-1,\alpha}(\overline{ B_{1/2}^+})}\le C
\ee
where in the last inequality we exploit the improved regularity along tangential derivatives in \eqref{eq:derivativeiv}. Since we obtained that $\varphi$ is $C^{k,\alpha}$ in $\overline{B_{1/2}^+}$, it remains to transfer this regularity to $v_d$. Although $\varphi$ is closely related to $v_d$, a separate argument is required to conclude the same regularity for $v_d$.

By the uniform convexity of $F$, we know that $A(p)e_d\cdot e_d\ge\lambda>0$ for every $p\in E$. 
Then, since $DF$ is $C^{k,\alpha}$-regular, being that $E$ is a connected bounded open set, by the global implicit function theorem there exists $H\in C^{k,\alpha}(E)$ such that 
$$
t=D F(p',p_d)\cdot e_d\iff p_d=H(p',t),\,\quad \mbox{for every } p \in E,\ t\in DF(E).
$$ 
Notice that the $C^{k,\alpha}$ norm of $H$ depends only on the $C^{k+1,\alpha}$ norm of $F$. Now, being $\varphi=DF(\nabla v)\cdot e_d \in C^{k,\alpha}(\overline{B_{1/2}^+})$ and $\nabla_{x'}v \in C^{k,\alpha}(\overline{B_{1/2}^+};\R^{d-1})$, we infer that $$v_d=H(\nabla_{x'}v,\varphi)\in C^{k,\alpha}(\overline{B_{1/2}^+}).$$ 
Moreover, by \eqref{eq:derivativeiv}, \eqref{derivativevarphii} and \eqref{derivativevarphid}, the following estimate holds
$$\|v_d\|_{C^{k,\alpha}(\overline{B_{1/2}^+})}\le C\sum_{i=1}^{d-1}\|v_i\|_{C^{k,\alpha}(\overline{B_{1/2}^+})}+\|\varphi\|_{C^{k,\alpha}(\overline{B_{1/2}^+})}\le C,$$
concluding the proof.
\end{proof}
\begin{proof}[Proof of \cref{t:shauder-quasi}]
    It follows immediately from \cref{prop:quasi}.
\end{proof}
\subsection{Proof of \texorpdfstring{$C^\infty$}{Cinfty} regularity} Now we prove the main results of this section.
\begin{proof}[Proof of \cref{prop:k-implies-k+1}]
Let $\gamma \in (-2,2)$ and $w$ be a regular solution, in the sense of \cref{def:definition2.1}. Up to translation and rescaling, assume that $0\in\partial\Omega_u$ and $\nabla w(0)=e_d$. Therefore, as in \Cref{s:subs.hodograph}, we can infer the existence of $\rho>0$ and $\delta>0$ such that 
$$
\Phi:\overline\Omega_w\cap B_\rho\to \R^d\cap\{x_d\ge0\},\quad \Phi(x',x_d):=(x', w(x',x_d)),\quad 
$$ 
is bijective from $\overline\Omega_w\cap B_\rho$ into $B_\delta^+$. Hence, there exists the hodograph transform $h\in C^{1,\alpha}(\overline{B_\delta^+})$ of $w$ and, by \cref{l:hsolves}, it is a solution of \eqref{eq:quasi-v}, with $\nabla h(0)=e_d$.\\
We proceed by showing that the regularity of $h$ can be improved. 
First, we observe that $$
A(p):=D^2 F(p)=\frac{2}{p_d}\mathrm{Id}-2\frac{e_d\otimes p+ p\otimes e_d}{p_d^2}+\frac{2(1+|p|^2)}{p_d^3}e_d\otimes e_d,
$$
is uniformly elliptic in $B_\eta(e_d)$, for some $\eta>0$. In fact, since $A(e_d)=2\mathrm{Id}$, for $\eta>0$ sufficiently small, we have that $$
-\mathrm{Id}\leq A(p)-2\mathrm{Id}\leq  \mathrm{Id},\quad\mbox{for every }p \in B_\eta(e_d).
$$
Moreover, since $\nabla h(0)=e_d$ and $A(\nabla h(0))=2\mathrm{Id}$, by the $C^{1,\alpha}$-regularity of $h$ in $\overline{B_\delta^+}$ we deduce the existence of a radius $\sigma>0$, possibly smaller then $\delta$, such that $$(\nabla h)(\overline{B_\sigma^+})\subset B_\eta(e_d).$$ Therefore, $F$ is uniformly convex in $E:=B_\eta(e_d)$ and, up to rescaling, the hodograph map $h$ fulfills the hypotheses of \cref{t:shauder-quasi}. Since $F$ is $C^{\infty}$ regular in $E$, then $h$ is $C^{\infty}$ in a neighborhood of the origin. Finally, by exploiting the inverse of the hodograph transformation we deduce the existence of $r>0$ such that $w\in C^{\infty}(\overline\Omega_w\cap B_{r})$. Similarly, since the derivatives of the graph of $\partial \O_w$ are given by the tangential derivatives of $h$ (see \eqref{e:deriv.graph}), the improved regularity of the hodograph transform implies that $\partial \Omega_w$ is smooth in $B_{r}$.
\end{proof}
The following result is a direct consequence of \cref{prop:k-implies-k+1}.
\begin{proof}[Proof of \cref{t:main-AP}]
The proof directly follows by \cref{prop:k-implies-k+1}. Indeed, 
let $u$ be a local minimizer of $\mathcal{J}_\gamma$ in $B_1$ and let $x_0\in\partial\Omega_u$ be a regular free boundary point. Without loss of generality, suppose that $x_0=0$ and consider the auxiliary function $w:=\beta u^{1/\beta}$. Then $w$ is a local minimizer to $\mathcal{E}_s$, and, up to rescaling, by \cref{l:fbcondition} it is a regular solution.
Thus, by \cref{prop:k-implies-k+1}, there exists $r>0$ small such that 
$$
\text{$w\in C^{\infty}(\overline\Omega_w\cap B_r)$ and the free boundary $\partial\Omega_w$ is smooth in $B_r$.} 
$$
Finally, the last part of the result follows by exploiting Hadamard's Lemma. Indeed, up to localize the problem in a smaller ball $B_r$ in such a way the boundary normal coordinates are well-defined,
we can deduce that $w(x)=\mathrm{dist}(x,\partial \O_w)\psi(x)$ where $\psi \in C^\infty(\overline{\O}_w\cap B_r)$. On the other hand, being $|\nabla w|=1$ on $\partial \O_w$, we get $\psi\equiv 1$ on $\partial \O_w\cap B_r$. Therefore, $\psi$ is positive in $\partial \O_w \cap B_r$ 
and thus
$$
\frac{u}{\mathrm{dist}(x,\partial \O_w)^\beta} = \left(\frac{1}{\beta}\frac{w}{\mathrm{dist}(x,\partial \O_u)}\right)^\beta=\left(\frac{1}{\beta}\psi(x)\right)^\beta
$$
is smooth in $\overline{\O}_w\cap B_r$.
\end{proof}

\section{The stability condition}\label{s:variations}
In this section we introduce the definitions of stable solutions for the Alt-Phillips problem. Then, we compute both the first and the second variations with respect to inner variations. Ultimately, we derive the rigidity condition of \cref{t:sing.min.cones} for minimizing (and stable) cones.

\subsection{Minimizers and stable solutions}\label{s:min-stable}
In order to compute the stability condition, we can consider 
solutions of the problem whose free boundary is smooth
around the points we want to deal with. Nevertheless, in light of the result of \cref{s:smooth}, it is not restrictive to work with regular solutions $w$ in the sense of \cref{def:definition2.1}, since their free boundary is smooth near the points of interest and the free boundary conditions is pointwise satisfied. 

Let $\gamma \in (-2,0)$, and $B$ be a ball in $\R^d$, then we define 
$$
\mathcal{J}_\gamma(u,B) := \int_{B}|\nabla u|^2 + u^\gamma \ind_{\{u>0\}}\,dx,\qquad \mathcal{E}_s(w,B):=\int_{B} w^s (|\nabla w|^2 +1) \ind_{\{w>0\}}\,dx.
$$
In the following we introduce the notion of stable solutions. 
In general, stability is defined for critical points of a variational functional with respect to a class of variations. However, for the Alt-Phillips problem with negative exponents, criticality must be replaced by a stronger assumption, as the validity of \eqref{e:hodo-pre-transform} cannot be directly deduced from domain variations (see \cref{l:useful} and the discussion in \cite{DeSilvaSavinNegativePower}). Throughout the paper, we address this by assuming that solutions are regular in the sense of \cref{def:definition2.1} near regular points.

\begin{definition}[Stable solutions]\label{def:definition3.1}
     Let $\gamma \in (-2,0)$, and $B$ be a ball in $\R^d$. We say that $u: B \to \R$ is a stable solution for $\mathcal{J}_\gamma$ in $B$, 
if 
\bea\label{e:stable-def-u}
\delta^2 \mathcal{J}_\gamma(u,B)[\xi]:=\frac{d^2}{d t^2}\bigg\lvert_{t=0}\mathcal{J}_\gamma(u \circ \Phi_t^{-1},B)\geq 0,\quad \mbox{where }\Phi_t(x):=x+t\xi(x),
\eea
for every $\xi \in C^{\infty}_c(B;\R^d)$. Equivalently, $w: B\to \R$ is stable solution for $\mathcal{E}_s$ in $B$,
if 
\bea\label{e:stable-def-w}
   \delta^2 \mathcal{E}_s(w,B)[\xi]:=\frac{d^2}{d t^2}\bigg\lvert_{t=0}\mathcal{E}_s(w \circ \Phi_t^{-1},B)\geq 0,\quad \mbox{where }\Phi_t(x):=x+t\xi(x),
   \eea
   for every $\xi \in C^{\infty}_c(B;\R^d)$.
\end{definition}
We also introduce the notion of local/global minimizers and global stable solutions.
\begin{definition}\label{def:definition3.2}
    Let $B$ be a ball in $\R^d$, we say that $u: B\to \R$ is a local minimizer of $\mathcal{J}_\gamma$ in $B$, if
    $$
    \mathcal{J}_\gamma(u,B)\leq \mathcal{J}_\gamma(v, B),\qquad \text{for all $v\in H^1(B)$ such that $v-u \in H^1_0(B)$}. 
    $$
    Moreover, $u: \R^d\to \R$ is said to be a global minimizer of $\mathcal{J}_\gamma$ in $\R^d$ (resp.~global stable solution in $\R^d$) if $u$ is a local minimizer (resp.~stable solution) in every ball $B\subset \R^d$. The previous definitions can be naturally extended to the function $w$.
\end{definition}

\subsection{Second inner variation} The following result, is our first characterization of the stability condition, which ultimately implies \cref{t:sing.min.cones}.

\begin{proposition}\label{prop:stability-cond}
Let $\gamma \in (-2,0)$, and $B$ be a ball in $\R^d$. Let $f \in C^\infty_c(B)$ and $u:B\to\R$ be a local minimizer of $\mathcal{J}_\gamma$ in $B$; or a stable solution in $B$, in the sense of \cref{def:definition3.1}. Assume that $u$ is regular in $\text{supp}(f)$, in the sense of \cref{def:definition2.1}.
Then 
$$
\int_{ \O_u} |\nabla u|^2 \left(|\nabla f|^2 - \mathcal{A}^2_u\, f^2\right)\,dx\geq 0,\quad\text{where }\mathcal{A}^2_u:= \frac{|\nabla^2 u|^2}{|\nabla u|^2} - \frac{|\nabla^2 u\nabla u|^2}{|\nabla u|^4}.
$$
Equivalently, let $w:B\to\R$ be a local minimizer of $\mathcal{E}_s$ in $B$; or a stable solution in $B$ in the sense of \cref{def:definition3.1}. Assume that $u$ is regular in $\text{supp}(f)$, in the sense of \cref{def:definition2.1}. Then 
$$
\int_{ \O_w} w^s|\nabla w|^2 \left(|\nabla f|^2 - \mathcal{A}^2_w\, f^2\right)\,dx\geq 0,\quad\text{where }\mathcal{A}^2_w:= \frac{|\nabla^2 w|^2}{|\nabla w|^2} - \frac{|\nabla^2 w\nabla w|^2}{|\nabla w|^4}.
$$
\end{proposition}
In order to compute a second order expansion close to regular free boundaries, we work directly with local minimizers $w$ of $\mathcal{E}_s$ (resp.~stable solutions) being smooth close to regular points. Through the section, given a vector field $\xi:\R^d \to \R^d$ we denote with $D\xi \in \R^{d\times d}$ the matrix with entries $(D\xi)_{ij}:=\partial_i \xi_j$, for every $1\leq i,j\leq d$.\\\vspace{-0.35cm}

In the following result we compute the first variation of the functional $\mathcal{E}_s$ along inner variations. Ultimately, we observe that regular solutions are indeed stationary along non-tangential variations. We stress that such computation is not necessary for the proof of \cref{prop:stability-cond}.
\begin{lemma}\label{l:useful}
Let $\gamma \in (-2,0)$, and $B$ be a ball in $\R^d$. Let $w \in H^1(B)$, then for every $\xi \in C^\infty_c(B;\R^d)$ we set as first variation of $\mathcal{E}_s$ at $w$ with respect to $\xi$ the quantity
$$
\delta \mathcal{E}_s(w,B)[\xi]:=\frac{d}{d t}\bigg\lvert_{t=0}\mathcal{E}_s(w \circ \Phi_t^{-1},B),\quad \mbox{where }\Phi_t(x):=x+t\xi(x).
$$
Then
\be\label{eq:tesi1}
\delta \mathcal{E}_s(w,B)[\xi]=\int_{\O_w}w^s\Big((|\nabla w|^2+1)\dive \xi-2D\xi\nabla w\cdot \nabla w\Big)\,dx.
\ee
Moreover, for every $f\in C^\infty_c(B)$, if $w$ is a regular solution in $\text{supp}(f)$, in the sense of \cref{def:definition2.1}, we have
\be\label{eq:tesi2}
\delta \mathcal{E}_s(w,B)\left[ \frac{\nabla w}{|\nabla w|} f\right] = 0.
\ee
\end{lemma}

\begin{proof}
Let $\xi \in C^\infty_c(B;\R^d)$ be a smooth vector field with compact support in $B$ and $\Phi_t:=\mathrm{Id}+t\xi$.
Let $w_t:B\to \R$ be such that
$w_t(x):=w(\Phi^{-1}_t(x))$.
Then, by the change of variable $y=\Phi_t(x)$, we get \be\label{eq:richiama1}\begin{aligned}
\mathcal{E}_s(w \circ \Phi_t^{-1},B) &=\int_{\Omega_{w_t}}w_t^s(|\nabla w_t|^2+1)\,dx\\
&=\int_{\Omega_w} w^s(\Phi_t^{-1}(x))\Big(\big|D\Phi_t^{-1}(x)\nabla w_t(\Phi_t^{-1}(x))\big|^2+1\Big)\,dx\\&=\int_{\Omega_w}w^s(y)\Big(D\Phi_t^{-1}(\Phi_t(y))^T D\Phi_t^{-1}(\Phi_t(y))\nabla w(y)\cdot \nabla w(y)+1\Big)|\text{det}D\Phi_t(y)|\,dy\\&=\int_{\Omega_w}w^s(y)\Big([D\Phi_t(y)]^{-T}[D\Phi_t(y)]^{-1}\nabla w(y)\cdot \nabla w(y)+1\Big)|\text{det}D\Phi_t(y)|\,dy.
\end{aligned}\ee
Using the facts that 
$$
D\Phi_t^{-1}=\mathrm{Id}-tD\xi + o(t)\qquad\text{and}\qquad|\text{det}D\Phi_t|=1+t\dive \xi+ o(t),
$$ 
we can expand the previous identity to first order in $t$, obtaining \eqref{eq:tesi1},
as we claimed.\\\vspace{-0.35cm}

Now, suppose that $w$ is a regular solution in $\text{supp}(f)$, in the sense of \cref{def:definition2.1}. By exploiting the $C^{1,\alpha}$-regularity of both $w$ and $\partial \O_w$ (for some $\alpha>-s$) we can rewrite \eqref{eq:tesi1} by integrating by parts. By the Rellich-type identity
 $$\dive\Big(|\nabla w|^2\xi-2(\xi\cdot\nabla w)\nabla w\Big)=|\nabla w|^2\dive \xi-2D\xi\nabla w\cdot \nabla w-2\Delta w(\xi\cdot \nabla w),$$ 
 we get
\bea 
\delta \mathcal{E}_s(w,B)[\xi]&=
\int_{\Omega_w}\left[w^s\dive\Big((|\nabla w|^2+1)\xi-2(\xi\cdot\nabla w)\nabla w\Big)+2w^s\Delta w(\xi\cdot \nabla w)\right]\,dx\\
&=
\int_{\O_w} \mathrm{div}\Big( w^s\,(|\nabla w|^2+1)\xi-2w^s(\xi\cdot\nabla w)\nabla w\Big)\, dx\\
&\quad+2\int_{\Omega_w} w^s\Delta w(\xi\cdot \nabla w) - \frac{s}{2}w^{s-1}\nabla w\cdot\Big((|\nabla w|^2+1)\xi-2(\xi\cdot\nabla w)\nabla w\Big)\,dx \\
&=
\int_{\O_w} \mathrm{div}\Big( w^s\,(1-|\nabla w|^2)\xi+2w^s|\nabla w|^2\xi-2w^s(\xi\cdot\nabla w)\nabla w\Big)\, dx\\
&\quad+2\int_{\Omega_w} w^s\left(\Delta w - \frac{s}{2}\frac{1-|\nabla w|^2}{w} \right)(\xi \cdot \nabla w) \,dx.
\eea
By substituting \eqref{e:hodo-pre-transform}, we already know that the last integral is null for every smooth vector field with compact support. On the other hand, for non-tangential variations, the previous computation can be further improved. Thus, given $f\in C^\infty_c(B)$, we  consider 
$$
\xi := \frac{\nabla w}{|\nabla w|} f.
$$
Indeed, since in $\partial \O_w \cap B$ the interior normal coincides with the gradient of $w$, the previous variation is perpendicular to $\partial \O_w$ as well. Then, by direct substitution, we have
\begin{align*}
\delta \mathcal{E}_s(w,B)\left[ \frac{\nabla w}{|\nabla w|}f\right] &= \int_{\O_w} \mathrm{div}\bigg( w^s\,(1-|\nabla w|^2)\frac{\nabla w}{|\nabla w|}f\bigg)\, dx =\int_{\partial\O_w} w^s\,(|\nabla w|^2-1)f \, d\mathcal{H}^{d-1}
\end{align*}
where in the last equality we exploit the Gauss-Green formula as in \cite[Proposition 2.7]{gaussgreen}. Precisely, since $w \in C^{1,\alpha}(\overline{\O}_w\cap \text{supp}(f))$ for some $\alpha>-s$, it implies that $w^s(|\nabla w|^2-1)$ extends continuously to zero on the free boundary $\partial \O_w$. Finally, \eqref{eq:tesi2} follows by the free boundary condition in \eqref{e:hodo-pre-transform}. 
\end{proof}

We proceed now to the study of the second variation.
\begin{lemma}\label{lemma:computation-2-var}
Let $\gamma \in (-2,0)$, and $B$ be a ball in $\R^d$. Let $f\in C^\infty_c(B)$ and $w:B\to\R$ be a regular solution in $\text{supp}(f)$, in the sense of \cref{def:definition2.1}.
Then
$$\delta^2\mathcal{E}_s(w,B)\left[\frac{\nabla w}{|\nabla w|}f\right]=\int_{\Omega_w}(f_1+f_2+f_3)\,dx,$$ where 
$$f_1:=\frac12w^s\Big(|\nabla w|^2+1\Big)\dive\left[\left(\frac{\Delta w\nabla w}{|\nabla w|^2}-\frac12\frac{\nabla(|\nabla w|^2)}{|\nabla w|^2}\right)f^2
\right],
$$
\bea f_2&:=w^s|\nabla f|^2|\nabla w|^2,
\eea
$$f_3:=w^s\Big((\Delta w)^2-|\nabla^2w|^2\Big)f^2+w^s\dive\left[\left(\frac12\nabla(|\nabla w|^2)-\Delta w\nabla w\right)f^2\right].$$
\end{lemma}
\begin{proof}
   Given $\xi \in C^\infty_c(B;\R^d)$ and $\Phi_t:=\mathrm{Id} + t \xi$, we consider $w_t:B\to \R$ be such that 
$w_t(x):=w(\Phi^{-1}_t(x))$. 
By \eqref{eq:richiama1} and the following second order expansions
\begin{align*}
D\Phi_t^{-1}&=\mathrm{Id}-tD\xi +t^2[D\xi]^2+ o(t^2),\\|\text{det}D\Phi_t|&=1+t\dive \xi+ \frac{t^2}{2}\Big((\dive \xi)^2-\text{Tr}(D\xi)^2\Big)+o(t^2),
\end{align*}
we have that 
\bea \delta^2 \mathcal{E}_s(w)[\xi]=\int_{\Omega_w} w^s\bigg(\Big(|\nabla w|^2+1\Big)&\frac{(\dive \xi)^2-\text{Tr}(D\xi)^2}{2}+|D\xi\nabla w|^2\\&\quad+2\Big([D\xi]^2\nabla w\cdot \nabla w-D\xi\nabla w\cdot \nabla w \dive \xi\Big)\bigg)\,dx.\eea
Now, set
$$f_1:= w^s\Big(|\nabla w|^2+1\Big)\frac{(\dive \xi)^2-\text{Tr}(D\xi)^2}{2},$$ $$f_2:=w^s|D\xi \nabla w|^2,
$$
$$f_3:=2w^s\Big([D\xi]^2\nabla w\cdot \nabla w-D\xi\nabla w\cdot \nabla w \dive \xi\Big).$$
Given $f\in C^{\infty}_c(B)$, we take $\xi:=\frac{\nabla w}{|\nabla w|}f$.
For what it concerns $f_1$ and $f_3$, we proceed by using some general differential identities, already exploited in \cite[Subsection 4.2]{thomas} for the study of the Alt-Phillips functional for positive exponents. In fact, by following their notations, the new formulations of $f_1$ and $f_3$ follow by substituting $\varphi=f|\nabla w|$ in to the proof of \cite[Theorem 4.1]{thomas}. Regarding $f_2$, we have that 
\bea 
\partial_{i}\xi_j&=\frac{\partial_{ij}w}{|\nabla w|}f+\frac{\partial_if\partial_jw}{|\nabla w|}-\frac{1}{|\nabla w|^2}\partial_j w \frac{(\nabla\partial_iw\cdot \nabla w)}{|\nabla w|}f.
\eea
Therefore
\bea D\xi\nabla w 
=\frac{\nabla ^2w\nabla w}{|\nabla w|}f+|\nabla w|\nabla f-\frac{1}{|\nabla w|^2}|\nabla w|^2\frac{\nabla^2w\nabla w}{|\nabla w|} f=|\nabla w|\nabla f,
\eea
which implies that $|D\xi\nabla w|^2=|\nabla w|^2|\nabla f|^2,$ concluding the proof.
\end{proof}

\begin{lemma}\label{lemma:computation-2-var2}
Let $\gamma \in (-2,0)$, and $B$ be a ball in $\R^d$. Let $f\in C^\infty_c(B)$ and $w:B\to\R$ be a local minimizer of $\mathcal{E}_s$ in $B$; or a stable solution in $B$, in the sense of \cref{def:definition3.1}. Assume that $w$ is regular in $\text{supp}(f)$, in the sense of \cref{def:definition2.1}. If $f_1$ and $f_3$ are as in \cref{lemma:computation-2-var}, then
$$\int_{\Omega_w}(f_1+f_3)\,dx=-\int_{\Omega_w}w^s|\nabla w|^2\mathcal{A}_w^2f^2\,dx,\quad\mbox{where } \mathcal{A}_w^2=\frac{|\nabla^2w|^2}{|\nabla w|^2}-\frac{|\nabla^2w\nabla w|^2}{|\nabla w|^4}.
$$ 
\end{lemma}
\begin{proof}
Since $w$ is a regular solution in $\text{supp}(f)$ with $C^{1,\alpha}$-regular free boundary in $\text{supp}(f)$, then, by \cref{prop:k-implies-k+1}, 
$$
\text{$w\in C^{\infty}(\overline\Omega_w\cap \text{supp}(f))$ and the free boundary $\partial\Omega_w$ is smooth in $\text{supp}(f)$.} 
$$
Expanding the divergence of the second term in $f_3$, we get \bea f_3&:=w^s\Big((\Delta w)^2-|\nabla^2w|^2\Big)f^2+w^s\dive\left[\left(\frac{\nabla(|\nabla w|^2)}2-\Delta w\nabla w\right)f^2\right]\\
&=w^s\Big((\Delta w)^2-|\nabla^2w|^2\Big)f^2+w^s \frac{\Delta (|\nabla w|^2)}{2}f^2+w^s\frac{\nabla(|\nabla w|^2)}2\cdot \nabla f^2- w^s\mathrm{div}\Big((\Delta w \nabla w)f^2\Big)\\&=w^s\Big((\Delta w)^2-|\nabla^2w|^2\Big)f^2+w^s \frac{\Delta (|\nabla w|^2)}{2}f^2+\dive\left(w^s\frac{\nabla(|\nabla w|^2)}2f^2\right)\\&\qquad-\dive\left(w^s\frac{\nabla(|\nabla w|^2)}2\right)f^2- w^s\mathrm{div}\Big((\Delta w \nabla w)f^2\Big).
\eea
Since $$\dive\left(w^s\frac{\nabla(|\nabla w|^2)}{2}\right)=\frac12\nabla w^s\cdot\nabla(|\nabla w|^2) +\frac12w^s \Delta (|\nabla w|^2),$$ then 
\bea f_3&=w^s\Big((\Delta w)^2-|\nabla^2w|^2\Big)f^2+\dive\left(w^s\frac{\nabla(|\nabla w|^2)}2f^2\right)
-\frac12\nabla w^s\cdot \nabla (|\nabla w|^2)f^2\\&\qquad- w^s\mathrm{div}\Big((\Delta w \nabla w)f^2\Big).
\eea
Then we can rewrite $f_1+f_3=f_0+f_{\text{parts}},$ where
\bea f_0:=w^s\Big((\Delta w)^2-|\nabla^2w|^2\Big)f^2-\frac12\nabla w^s\cdot \nabla (|\nabla w|^2)f^2
\eea
and
\bea
f_{\text{parts}} &:=\frac12 w^s \big(|\nabla w|^2 +1\big)\mathrm{div}\left[\left( \frac{\Delta w\nabla w}{|\nabla w|^2} - \frac12 \frac{\nabla (|\nabla w|^2)}{|\nabla w|^2}\right)f^2\right]
+\mathrm{div}\left(w^s  \frac{\nabla(|\nabla w|^2)}{2}f^2\right) \\&\qquad -  w^s\mathrm{div}\Big((\Delta w \nabla w)f^2\Big). \eea
By rewriting $f_{\text{parts}}$, we get 
\bea
f_{\text{parts}}&=\dive\left[w^s\left(\big(|\nabla w|^2+1\big)\left(\frac12\frac{\Delta w\nabla w}{|\nabla w|^2}-\frac14 \frac{\nabla (|\nabla w|^2)}{|\nabla w|^2}\right)-\Delta w\nabla w+\frac{\nabla(|\nabla w|^2)}{2}\right)f^2\right]\\&\qquad+ \int_{\Omega_w} \left[ - \frac12 \left(\frac{\Delta w \nabla w}{|\nabla w|^2} - \frac12 \frac{\nabla (|\nabla w|^2)}{|\nabla w|^2}\right)\nabla \big(w^s \big(1+|\nabla w|^2\big)\big)+\Delta w \nabla w \cdot \nabla w^s\right]f^2
\\&= \mathrm{div}(F)\\
& \qquad+ \left[ - \frac12 \left(\frac{\Delta w \nabla w}{|\nabla w|^2} - \frac12 \frac{\nabla (|\nabla w|^2)}{|\nabla w|^2}\right)\nabla \big(w^s \big(1+|\nabla w|^2\big)\big)+\Delta w \nabla w \cdot \nabla w^s\right]f^2
\eea
where 
$$
F:= \frac12  w^s(1-|\nabla w|^2)\left(\frac{\Delta w \nabla w}{|\nabla w|^2} - \frac{\nabla^2 w \nabla w}{|\nabla w|^2}\right)f^2.
$$
Therefore 
\bea 
f_{\text{parts}}-\dive(F)&= \bigg[- \frac12 \left(\frac{\Delta w \nabla w}{|\nabla w|^2} - \frac12 \frac{\nabla (|\nabla w|^2)}{|\nabla w|^2}\right)\Big(\nabla w^s  \big(1+|\nabla w|^2\big)+w^s\nabla (|\nabla w|^2)\Big)\\&\qquad+\Delta w \nabla w \cdot \nabla w^s\bigg]f^2\\&=
\frac12 \frac{\Delta w \nabla w\cdot\nabla w^s}{|\nabla w|^2}\big(|\nabla w|^2-1\big)f^2-\frac12w^s\frac{\Delta w \nabla w\cdot\nabla(|\nabla w|^2)}{|\nabla w|^2}f^2\\&\qquad+\frac14\frac{\nabla (|\nabla w|^2)\cdot\nabla w^s}{|\nabla w|^2}\big(1+|\nabla w|^2\big)f^2
+\frac14w^s \frac{|\nabla (|\nabla w|^2)|^2}{|\nabla w|^2}f^2\\&=
-w^s (\Delta w)^2 f^2 
-\frac12w^s\frac{\Delta w \nabla w\cdot\nabla(|\nabla w|^2)}{|\nabla w|^2}f^2\\&\qquad+\frac14\frac{\nabla (|\nabla w|^2)\cdot\nabla w^s}{|\nabla w|^2}\big(1+|\nabla w|^2\big)f^2
+w^s \frac{|\nabla^2w\nabla w|^2}{|\nabla w|^2}f^2
\eea
where in the last equality we used the equation for $w$ (see \eqref{e:hodo-pre-transform}).
Combining the above identities, we have
\bea f_0+f_{\text{parts}}&=-w^s\left(|\nabla^2w|^2-\frac{|\nabla^2w\nabla w|^2}{|\nabla w|^2}\right)f^2-\frac12\nabla w^s\cdot \nabla (|\nabla w|^2)f^2\\&\qquad-\frac12w^s\frac{\Delta w \nabla w\cdot\nabla(|\nabla w|^2)}{|\nabla w|^2}f^2+\frac14\frac{\nabla (|\nabla w|^2)\cdot\nabla w^s}{|\nabla w|^2}\big(1+|\nabla w|^2\big)f^2+\dive(F)\\&=-w^s\left(|\nabla^2w|^2-\frac{|\nabla^2w\nabla w|^2}{|\nabla w|^2}\right)f^2
-\frac12w^s\frac{\Delta w \nabla w\cdot\nabla(|\nabla w|^2)}{|\nabla w|^2}f^2\\&\qquad+\frac14\frac{\nabla (|\nabla w|^2)\cdot\nabla w^s}{|\nabla w|^2}\big(1-|\nabla w|^2\big)f^2+\dive(F)
\eea
and so
$$
f_0+f_{\text{parts}}=-w^s\left(|\nabla^2w|^2-\frac{|\nabla^2w\nabla w|^2}{|\nabla w|^2}\right)f^2+\dive(F),
$$
where in the last equality we used the PDE satisfied by $w$ (see \eqref{e:hodo-pre-transform}). Hence, as in the proof of \cref{l:useful} we can apply the Gauss-Green formula as in \cite[Proposition 2.7]{gaussgreen}. In fact, since $w\in C^{\infty}(\overline\Omega_w\cap \text{supp}(f))$, the vector field $F$ can be extended continuously to zero on the free boundary $\partial\Omega_w$ in $\text{supp}(f)$ and so $F$ is divergence free in $\O_w\cap \text{supp}(f)$.
Then the result follows integrating the last identity and using the definition of $\mathcal{A}_w^2$.
\end{proof}
We conclude with the proof of the stability condition in \cref{t:sing.min.cones} and \cref{prop:stability-cond}.
\begin{proof}[Proof of \cref{prop:stability-cond}]
    It follows immediately by combining \cref{lemma:computation-2-var} and \cref{lemma:computation-2-var2}. 
\end{proof}
\begin{proof}[Proof of \cref{t:sing.min.cones}]
Since the free boundary $\partial \O_w$ is a smooth cone outside the origin, for every $f \in C^\infty_c(\R^d\setminus\{0\})$ we can apply \cref{prop:stability-cond} to conclude.
\end{proof}
\begin{remark}[Stability condition for stable solutions]\label{rem:holds-for-global}
    Using \cref{prop:stability-cond}, we observe that \cref{t:sing.min.cones} holds even for global stable solutions $u$ of $\mathcal{J}_\gamma$,  regular outside the origin, in the sense of \cref{def:definition3.2} and \cref{def:definition2.1}. 
    Similarly, \cref{t:sing.min.cones} holds for $w$ under the corresponding assumptions.
\end{remark}
\begin{remark}[Stability conditions for general exponents]\label{rem:differences}
Recently in \cite{thomas}, the authors obtained a stability condition  for the Alt-Phillips problem with positive exponents, by computing the second variations of the form
$$
t \mapsto \mathcal{E}_s\left(w\circ (\mathrm{Id}+t \xi)^{-1}\right),\quad\mbox{where }\xi:=\frac{\nabla w}{|\nabla w|^2}\varphi,
$$
and $\varphi \in C^\infty_c(\R^d \setminus \{0\})$. 
Their condition has a different form from our \eqref{e:variations} and is given by \be\label{eq:stability-thomas}\int_{\Omega_w}w^s|\nabla \varphi|^2\,dx\ge\int_{\Omega_w}w^s\frac{\Delta w}{w}\varphi^2\,dx,\quad \mbox{for every }\varphi \in C^\infty_c(\R^d \setminus \{0\}).\ee
The choice of such variations is natural and it coincides with the one considered by Caffarelli, Jerison and Kenig in \cite{CaffarelliJerisonKenig04:NoSingularCones3D}, for the Alt-Caffarelli problem ($\gamma =0$), where they show that 
$$
\int_{\O_w}|\nabla \varphi|^2\,dx\geq \int_{\partial \O_w}H\varphi^2\,d\mathcal{H}^{d-1},\quad \mbox{for every }\varphi \in C^\infty_c(\R^d \setminus \{0\}).
$$
Nevertheless, it is well-known that such latter inequality, can be also written in a Sternberg-Zumbrun form \cite{sz98}. Indeed, by choosing $\varphi:=|\nabla w|f$ and integrating by parts, we get 
$$
\int_{ \O_w} |\nabla w|^2\left(|\nabla f|^2 - \mathcal{A}^2_w\, f^2\right)\,dx \geq 0,\quad \mbox{where }\mathcal{A}^2_w:= \frac{|\nabla^2 w|^2}{|\nabla w|^2} - \frac{|\nabla^2 w\nabla w|^2}{|\nabla w|^4},
$$
which is exactly \cref{t:sing.min.cones} in the case $s=0$ (i.e.,~$\gamma=0$).
Naturally, a similar formulation holds for the case of positive exponents.

Surprisingly, in the case $\gamma \in (-2,0)$, the presence of the singular weight $w^s$ prevents to write both the conditions. Heuristically, the stability condition can be computed by considering only variations of the form \eqref{e:variations}, which are more close to the one exploited in the study of stable minimal surfaces. Indeed, by comparing the stability condition \eqref{eq:ripresa} below with the one for positive exponents \eqref{eq:stability-thomas}, we observe that in the negative regime, it is not possible to split the integrals and integrate by parts the divergence term in \eqref{eq:ripresa}, since the weight $w^s$ is singular on $\partial \Omega_w$. Moreover, for negative exponents, such integration can be computed if and only if the mean curvature of $\partial \O_w$ is identically zero. 
\end{remark}

\section{Axially symmetric cones}\label{sec:axially}
In this section, we consider the specific case of minimizing axially symmetric cones, namely homogeneous minimizers $u$ that are invariant under rotations around a fixed axis. Precisely, we say that $u$ is axially symmetric if, up to a rotation, we have that for every $x =(x',x_d)\in \R^d$
$$u(x',x_d)=u(\tau,x_d),\quad\text{where } \tau:=|x'|.$$
The main result of the section is \cref{thm:axially}, namely that minimizing axially symmetric cones are one-dimensional in low dimensions.\\\vspace{-0.35cm}

Before presenting the proof of the main result, we begin by rephrasing the stability condition under axial symmetry as an Hardy-type inequality. For simplicity of notation, we set
$$u_\tau:=\nabla u\cdot \frac{(x',0)}{\tau}.$$ 
The same notations apply to $w:=\beta u^{1/\beta}$.
The following is the main result towards the proof of \cref{thm:axially}. 
\begin{proposition}\label{proposition:d-2}
    Let $\gamma \in (-2,0)$, $d\ge3$ and $u\in C^{0,\beta}(\R^d)$ be a $\beta$-homogeneous global minimizer of $\mathcal{J}_{\gamma}$ in $\R^d$; or a global stable solution in the sense of \cref{def:definition3.2}. Assume that $u$ is axially symmetric and that is regular outside the origin, in the sense of \cref{def:definition2.1}.
    Then
    \be\label{d-21}
    \int_{\Omega_u} u_\tau^2 \left(|\nabla \eta|^2-(d-2)\frac{\eta^2}{\tau^2}\right)\,dx\ge0,\ee for every $\eta\in C^\infty_c(\R^d)$. 
    
    Equivalently, let $w\in C^{0,1}(\R^d)$ be a $1$-homogeneous global minimizer of $\mathcal{E}_{s}$ in $\R^d$; or a global stable solution, in the sense of \cref{def:definition3.2}. Assume that $w$ is axially symmetric and that is regular outside the origin, in the sense of \cref{def:definition2.1}. 
    Then
    \be\label{d-22}
    \int_{\Omega_w}w^s w_\tau^2 \left(|\nabla \eta|^2-(d-2)\frac{\eta^2}{\tau^2}\right)\,dx\ge0,\ee for every $\eta\in C^\infty_c(\R^d)$.
\end{proposition}
The proof of \cref{proposition:d-2} is carried out by working directly with $w$ and proving \eqref{d-22}, from which \eqref{d-21} immediately follows. The key point is to use $f:=\frac{1}{|\nabla w|}w_\tau \eta $ as a test function in the stability condition of \cref{t:sing.min.cones}, where $\eta$ is a cut-off function with compact support in $\R^d$. We stress, in relation to \eqref{e:variations}, that such test function corresponds to testing the second variation of $\mathcal{E}_s$ with the vector field 
$\xi=\frac{\nabla w}{|\nabla w|^2}w_\tau \eta.
$\\
We proceed directly with the proof of \cref{proposition:d-2} (see also \cite{thomas} for a similar result in the context of positive exponents).
\begin{proof}[Proof of \cref{proposition:d-2}]
   Since $w$ is a regular solution in $\R^d\setminus\{0\}$ with $C^{1,\alpha}$-regular free boundary in $\R^d\setminus\{0\}$, then, by \cref{prop:k-implies-k+1}, 
$$
\text{$w\in C^{\infty}(\overline\Omega_w\cap (\R^d\setminus\{0\}))$ and the free boundary $\partial\Omega_w$ is smooth in $\R^d\setminus\{0\}$.} 
$$
   The proof is divided in three steps.\\ \vspace{-0.3cm}

    \noindent\textit{Step 1.} We start by proving that \be\label{eq:ripresa} \int_{\Omega_w}w^s|\nabla \varphi|^2\,dx-\int_{\Omega_w}\bigg[\dive\left(w^s\frac{\nabla(|\nabla w|^2)}{2|\nabla w|^2}\varphi^2\right)+w^s\frac{\Delta w}{w}\varphi^2\bigg]\,dx\ge0,
\ee 
    for every $\varphi \in C^\infty_c(\R^d\setminus \{0\})$, computing the following second variation  
$$
\delta^2 \mathcal{E}_s(w,\R^d)\left[\frac{\nabla w}{|\nabla w|^2}\varphi\right]\geq 0.
$$
Precisely, by substituting $f:=\frac{\varphi}{|\nabla w|}$ in the stability condition for $w$ in \cref{t:sing.min.cones},
 we get $$\int_{\Omega_w}w^s|\nabla w|^2\left|\nabla \left(\frac{\varphi}{|\nabla w|}\right)\right|^2\,dx-\int_{\Omega_w}w^s|\nabla w|^2\left(\frac{|\nabla^2w|^2}{|\nabla w|^2}-\frac{|\nabla^2w\nabla w|^2}{|\nabla w|^4}\right)\frac{\varphi^2}{|\nabla w|^2}\,dx\ge0.$$
Since \bea |\nabla w|^2\left|\nabla \left(\frac{\varphi}{|\nabla w|}\right)\right|^2&=|\nabla w|^2\left|\frac{\nabla \varphi}{|\nabla w|}-\frac{\varphi\nabla|\nabla w|}{|\nabla w|^2}\right|^2
\\&=|\nabla w|^2\left(\frac{|\nabla \varphi|^2}{|\nabla w|^2}+\varphi^2\frac{|\nabla |\nabla w||^2}{|\nabla w|^4}-\frac{\nabla (|\nabla w|)\cdot \nabla \varphi^2}{|\nabla w|^3}\right)
\\&=|\nabla \varphi|^2+\varphi^2\frac{|\nabla^2w\nabla w|^2}{|\nabla w|^4}-\frac{\nabla |\nabla w|}{|\nabla w|}\cdot \nabla \varphi^2,
\eea
then
\bea \int_{\Omega_w}w^s|\nabla \varphi|^2\,dx- \int_{\Omega_w}\left[w^s\frac{\nabla |\nabla w|}{|\nabla w|}\cdot \nabla \varphi^2+w^s\frac{|\nabla^2w|^2}{|\nabla w|^2}\varphi^2-2w^s\frac{|\nabla^2w\nabla w|^2}{|\nabla w|^4}\varphi^2 \right]\,dx\ge0.\eea
By applying the Leibnitz rule to the second term above, we get
\bea \int_{\Omega_w}w^s|\nabla \varphi|^2\,dx&-\int_{\Omega_w}\bigg[\dive\left(w^s \frac{\nabla |\nabla w|}{|\nabla w|}\varphi^2\right)-\dive\left(w^s\frac{\nabla|\nabla w|}{|\nabla w|}\right)\varphi^2+w^s\frac{|\nabla^2w|^2}{|\nabla w|^2}\varphi^2\\&\qquad-2w^s\frac{|\nabla^2w\nabla w|^2}{|\nabla w|^4}\varphi^2 \bigg]\,dx\ge0.
\eea
Expanding the third term above using that $$\dive\left(w^s\frac{\nabla|\nabla w|}{|\nabla w|}\right)=\dive\left(\frac{\nabla(|\nabla w|^2)}{2|\nabla w|^2}\right),$$ we have
\bea 
\int_{\Omega_w}w^s|\nabla \varphi|^2\,dx&-\int_{\Omega_w}\bigg[\dive\left(w^s\frac{\nabla|\nabla w|}{|\nabla w|}\varphi^2\right)
-\frac{\nabla w^s\cdot\nabla(|\nabla w|^2)}{2|\nabla w|^2}\varphi^2-w^s\frac{\Delta (|\nabla w|^2)}{2|\nabla w|^2}\varphi^2\\&\qquad+w^s\frac{|\nabla (|\nabla w|^2)|^2}{2|\nabla w|^4}\varphi^2+w^s\frac{|\nabla^2w|^2}{|\nabla w|^2}\varphi^2-2w^s\frac{|\nabla^2w\nabla w|^2}{|\nabla w|^4}\varphi^2\bigg]\,dx\ge0.\eea
Using the Bochner's identity and the PDE satisfied by $w$ (see \eqref{e:hodo-pre-transform}), we get
\begin{align*}
\frac12\Delta (|\nabla w|^2)&=|\nabla^2w|^2+\nabla(\Delta w)\cdot \nabla w\\
&=|\nabla^2w|^2- 
\frac{\Delta w}{w}|\nabla w|^2-\frac s2\frac{\nabla(|\nabla w|^2)\cdot \nabla w }{w}.
\end{align*} 
Moreover, since $$\frac{|\nabla (|\nabla w|^2)|^2}{2|\nabla w|^4}
=2\frac{|\nabla^2w\nabla w|^2}{|\nabla w|^4},$$
then we obtain \eqref{eq:ripresa}, concluding Step 1.
\\ \vspace{-0.3cm}
   
    \noindent\textit{Step 2.} 
    In the second step, we prove \eqref{d-22} for every $\eta\in C^\infty_c(\R^d)$ such that $\text{supp}(\eta)\subset \{\tau>0\}$.
    We choose $\varphi:=c\eta$ in the above inequality \eqref{eq:ripresa}, where $c:=w_\tau$. Since
\bea \int_{\Omega_w}w^s|\nabla (c\eta)|^2\,dx&=\int_{\Omega_w}\bigg[w^s|\nabla \eta|^2c^2+w^s|\nabla c|^2\eta^2+w^sc\nabla c\cdot\nabla \eta^2\bigg]\,dx\\&=\int_{\Omega_w}\bigg[w^s|\nabla \eta|^2c^2+w^s|\nabla c|^2\eta^2+\dive\big(w^s(c\nabla c) \eta^2\big)-\dive(w^sc\nabla c)\eta^2\bigg]\,dx\\&=\int_{\Omega_w}\bigg[w^s|\nabla \eta|^2c^2+\dive\big(w^s(c\nabla c) \eta^2\big)-w^sc\Delta c\eta^2-sw^{s-1}c\nabla w\cdot\nabla c\eta^2\bigg]\,dx,
\eea 
then \bea\label{eq:richiama0}\begin{aligned} 0&\le\int_{\Omega_w}\bigg[w^s|\nabla (c\eta)|^2\,dx-\dive\left(w^s\frac{\nabla(|\nabla w|^2)}{2|\nabla w|^2}c^2\eta^2\right)-w^s\frac{\Delta w}{w}c^2\eta^2\bigg]\,dx\\&=\int_{\Omega_w}\bigg[w^s|\nabla \eta|^2c^2-w^s\bigg(\Delta c+s\frac{\nabla w\cdot\nabla c}{w}+\frac{\Delta w}{w}c\bigg)c\eta^2
\\&\qquad+\dive\left(w^s(c\nabla c) \eta^2-w^s\frac{\nabla(|\nabla w|^2)}{2|\nabla w|^2}c^2\eta^2\right)\bigg]\,dx.
\end{aligned}\eea
By direct differentiation (see for instance \cite[Equation 6.4]{thomas}) we have that
$$\Delta c+s\frac{\nabla w\cdot\nabla c}{w}+\frac{\Delta w}{w}c=(d-2)\frac{c}{\tau^2}\quad\text{in } \Omega_w\cap\{\tau>0\},$$
then \bea\label{eq:richiama2}\begin{aligned} 0&\le\int_{\Omega_w}\bigg[w^s|\nabla \eta|^2c^2-(d-2)w^s\frac{c^2}{\tau^2}\eta^2+\dive\left(w^s(c\nabla c) \eta^2-w^s\frac{\nabla(|\nabla w|^2)}{2|\nabla w|^2}c^2\eta^2\right)\bigg]\,dx.
\end{aligned}\eea
The step is concluded once we prove that the third term above is zero. Hence, set $F$ to be the vector field defined as $$F:=(c\nabla c) \eta^2-\frac{\nabla(|\nabla w|^2)}{2|\nabla w|^2}c^2\eta^2,$$ then 
$$\int_{\Omega_w}\dive(w^sF)\,dx=\lim_{\eps\to0^+}\int_{\{w>\eps\}}\dive(w^sF)\,dx=\lim_{\eps\to0}\int_{\partial\{w>\eps\}}w^s (F\cdot \nu)\,d\HH^{d-1},$$
where $\nu=-\frac{\nabla w}{|\nabla w|}$. We observe that the function $F\cdot \nu$ is smooth in a neighborhood of $\partial\Omega_w \cap\text{supp}(\eta)$. Moreover, if $H$ is the mean curvature of $\partial \O_w$ pointing towards the complement of $\O_w$, by following the same argument in \cite[Equation 4.4]{xavixaviglobal}, we have
$$c_\nu=Hc\quad\text{on }\partial\Omega_w,\qquad \text{and}\qquad H=\frac{\nabla(|\nabla w|^2)}{2|\nabla w|^2}\cdot \nu\quad\text{on }\partial\Omega_w.$$
Thus, since $F\cdot \nu=0$ on $\partial\Omega_w$, there exists $\eps>0$ small enough, so that $$|(F\cdot \nu)| \le C\eps\quad\text{in a neighborhood of $\partial\Omega_w\cap \text{supp}(\eta)$},$$ for some constant $C>0$ depending on $\eta$ but independent of $\eps$.
Thus, since $1+s>0$, we obtain $$\lim_{\eps\to0^+}\left|\int_{\partial\{w>\eps\}}w^s (F\cdot \nu)\,d\HH^{d-1}\right|\le C\lim_{\eps\to0^+}\eps^{1+s}=0,$$ as we claimed.
\\ \vspace{-0.3cm}
   
    \noindent\textit{Step 3.} 
    The proof is concluded once we show that \eqref{d-22} holds for every $\eta\in C^\infty_c(\R^d)$. This extension follows by standard arguments (see e.g.,~\cite{xavixaviglobal,thomas}).  
    Precisely, given $\eta\in C^\infty_c(\R^d)$ we set $\eta_\eps:=\eta \zeta_\eps$, where $\zeta_\eps$ is first chosen as an axially symmetric cut-off function, and subsequently replaced by a radial one. 
    Precisely, if $\zeta\in C^\infty_c(\R)$ is the one dimensional cut-off function such that $\zeta\equiv0$ in $(-\infty,1/2)$ and $\zeta\equiv1$ in $(1,+\infty)$, then we set first $\zeta_\eps=\zeta(\tau/\eps)$ and subsequently $\zeta_\eps=\zeta(|x|/\eps)$.
    Since by regularity of $w$ and the axially symmetric property, we have that $w^sw_\tau^2\to 0$ as $\tau\to0^+$, then
    sending $\eps\to 0^+$ (here we are using the assumption that $d\ge 3$), we conclude the stability inequality \eqref{d-22} for $\eta$.
\end{proof}

Now we are ready to prove \cref{thm:axially}.
\begin{proof}[Proof of \cref{thm:axially}]
    The case $d=2$ follows by standard arguments (see e.g.,~\cite{FVstateofart}). Indeed, we can directly use the logarithmic cut-off function $$f(x):=\begin{cases}
        1&\quad \text{if }|x|\le1,\\
        \frac{\log R-\log|x|}{\log R}&\quad \text{if }1\le|x|\le R,\\
        0&\quad \text{if }|x|\ge R,\\
    \end{cases}$$ in the stability condition \eqref{e:stable1} of \cref{t:sing.min.cones}. The conclusion follwos by applying the Sternberg-Zumbrun formula \eqref{szformula} and sending $R\to+\infty$.

    From now on, we focus on the case $d\ge3$, which allows us to apply \cref{proposition:d-2}. We proceed by proving that $w$ is one-dimensional, which implies the same result on $u$.
    Observe that, by a standard approximation argument, the inequalities in \eqref{proposition:d-2} holds even for test function $\eta$ which are just Lipschitz and not $C^\infty$. Let $\theta>0$ a small constant to be chosen later. For every $\eps\in(0,1)$ and $R\ge1$, we define $\eta_{\eps,R}$ such that    
    $$\eta_{\eps,R}:=\begin{cases}
       \tau^{-\theta}\zeta_R&\quad \text{if }\tau>\eps,\\
       \eps^{-\theta}\zeta_R&\quad \text{if }\tau\le\eps,\\
   \end{cases}$$
where $\zeta_R\in C^\infty(\R^d)$ is a cut-off function such that $\zeta_R\ge0$, $\zeta_R\equiv1$ in $B_R$, $\zeta_R\equiv0$ in $B_{2R}$ and \be\label{eq:stima-1}|\nabla \zeta_R|\le \frac{C}{R}.\ee Since $w$ is $1$-homogeneous, we notice \be\label{eq:stima0}w^sw_\tau^2\le CR^s\quad\text{in }B_{2R}.
\ee
We compute first the two terms of \eqref{d-22} in \cref{proposition:d-2}. First, we have
\bea\int_{\R^d}w^s\frac{w_\tau^2}{\tau^2}\eta_{\eps,R}\,dx=\int_{B_{2R}\cap\{\tau>\eps\}}w^sw_\tau^2\tau^{-2\theta-2}\zeta_R^2\,dx+\int_{B_{2R}\cap\{\tau\le\eps\}}w^sw_\tau^2 \tau^{-2}\eps^{-2\theta}\zeta_R^2\,dx.\eea
   Secondly, since $$|\nabla \eta_{\eps,R}|^2\le\begin{cases}
       \theta^2\tau^{-2\theta-2}\zeta_R&\quad\text{in } B_R\cap \{\tau>\eps\},\\
       \theta^2\tau^{-2\theta-2}\zeta_R+\tau^{-2\theta}|\nabla \zeta_R|^2&\quad\text{in } \left(B_{2R}\setminus B_R\right)\cap \{\tau>\eps\},\\
       \eps^{-2\theta}|\nabla \zeta_R|^2&\quad\text{in } B_{2R}\cap\{\tau\le\eps\},\\
   \end{cases}$$
then
\bea \int_{\R^d}w^s w_\tau^2|\nabla \eta_{\eps,R}|^2\,dx&\le\theta^2\int_{B_{2R}\cap\{\tau>\eps\}}w^sw_\tau^2\tau^{-2\theta-2}\zeta_R^2\,dx+\int_{\left(B_{2R}\setminus B_R\right)\cap\{\tau>\eps\}}w^sw_\tau^2\tau^{-2\theta}|\nabla\zeta_R|^2\,dx\\&\qquad+\int_{B_{2R}\cap\{\tau\le\eps\}}w^sw_\tau^2\eps^{-2\theta}|\nabla \zeta_R|^2\,dx.\eea
Therefore, by \cref{proposition:d-2}, we have that
\be\label{eq:stima3}\begin{aligned}
(d-2)\int_{B_{2R}\cap\{\tau>\eps\}}w^sw_\tau^2\tau^{-2\theta-2}\zeta_R^2\,dx&\le
(d-2)\int_{B_{2R}\cap\{\tau>\eps\}}w^sw_\tau^2\tau^{-2\theta-2}\zeta_R^2\,dx\\&\qquad+(d-2)\int_{B_{2R}\cap\{\tau\le\eps\}}w^sw_\tau^2 \tau^{-2}\eps^{-2\theta}\zeta_R^2\,dx
\\&\le\theta^2\int_{B_{2R}\cap\{\tau>\eps\}}w^sw_\tau^2\tau^{-2\theta-2}\zeta_R^2\,dx \\&\qquad\int_{\left(B_{2R}\setminus B_R\right)\cap\{\tau>\eps\}}w^sw_\tau^2\tau^{-2\theta}|\nabla\zeta_R|^2\,dx\\&\qquad\qquad+\int_{B_{2R}\cap\{\tau\le\eps\}}w^sw_\tau^2\eps^{-2\theta}|\nabla \zeta_R|^2\,dx,
\end{aligned}\ee 
leading to
\bea\label{eq:stima2}\begin{aligned} (d-2-\theta^2)\int_{B_{2R}\cap\{\tau>\eps\}}w^sw_\tau^2\tau^{-2\theta-2}\zeta_R^2\,dx&\le \int_{\left(B_{2R}\setminus B_R\right)\cap\{\tau>\eps\}}w^sw_\tau^2\tau^{-2\theta}|\nabla\zeta_R|^2\,dx\\&\qquad+\int_{B_{2R}\cap\{\tau\le\eps\}}w^sw_\tau^2\eps^{-2\theta}|\nabla \zeta_R|^2\,dx.\end{aligned}
\eea
An estimate of the right-hand side is obtained by considering the two integrals separately.
\begin{itemize}
    \item[(i)] For the first integral, we change the coordinate $x\mapsto (\tau,t)$ with $\tau>\eps$ and $t\in(-2R,2R)$, so that $dx=\tau^{d-2}\,d\tau\,dt$. Then, using \eqref{eq:stima-1} and \eqref{eq:stima0}, we have   
    \bea\label{eq:stima4}\begin{aligned}\int_{\left(B_{2R}\setminus B_R\right)\cap\{\tau>\eps\}}w^sw_\tau^2\tau^{-2\theta}|\nabla\zeta_R|^2\,dx&\le CR^{s-2}\int_{\left(B_{2R}\setminus B_R\right)\cap\{\tau>\eps\}}\tau^{-2\theta}\,dx\\&\le CR^{s-1}\int_{\eps}^{2R}\tau^{d-2-2\theta}\,d\tau\\&=CR^{s-1}\left(R^{d-1-2\theta}-\eps^{d-1-2\theta}\right).
    \end{aligned}\eea
    \item[(ii)] For the second integral, by \eqref{eq:stima-1} and \eqref{eq:stima0}, we have \bea\label{eq:stima5}\int_{B_{2R}\cap\{\tau\le\eps\}}w^sw_\tau^2\eps^{-2\theta}|\nabla \zeta_R|^2\,dx\le CR^{s-2}\eps^{-2\theta}\int_{B_{2R}\cap\{\tau\le\eps\}}\,dx= CR^{s-1}\eps^{d-1-2\theta}.\eea
\end{itemize}
Combining \eqref{eq:stima3} with the last two estimates above, we infer that
\be\label{eq:stima6}(d-2-\theta^2)\int_{B_{2R}\cap\{\tau>\eps\}}w^sw_\tau^2\tau^{-2\theta-2}\zeta_R^2\,dx\le CR^{d+s-2-2\theta}+CR^{s-1}\eps^{d-1-2\theta}.\ee
Finally, the conclusion follows by choosing $\theta\in\R$ so that \be\label{eq:scelta-theta}d-2-\theta^2>0\quad\text{and}\quad d+s-2-2\theta<0.\ee 
Indeed, noting that $d-1-2\theta>\theta^2+1-2\theta\ge0$, we may first let $\eps\to0^+$ and then $R\to+\infty$ in \eqref{eq:stima6}, which, in view of \eqref{eq:scelta-theta}, yields $w_\tau\equiv0$. That is $w$ is one-dimensional. Such a choice of $\theta \in \R$ is possible provided that $$\sqrt{d-2}>\frac{d+s-2}{2}.$$ 
Equivalently, this condition holds if $$2+(1-\sqrt{1-s})^2<d<2+(1+\sqrt{1-s})^2.$$ 
The desired conclusion follows by rewriting the above inequality in terms of $\gamma$.
\end{proof}
\begin{remark}[On the homogeneity assumption]\label{rem:without-homo} 
In the proof of \cref{thm:axially}, the homogeneity assumption is in fact unnecessary. It was only used to derive the estimate \eqref{eq:stima0}. However, this bound actually follows from the regularity of the solution  $u\in C^{0,\beta}(\R^d)$, together with the identity $u_\tau^2=\beta^{-s}w^sw_\tau^2$.
\end{remark}

\section{Asymptotic of stable cones as \texorpdfstring{$\gamma\to-2$}{gammato-2}} 
\label{sec:asymptotic}
In this section, we begin by providing a variational characterization of the stability condition for homogeneous global solutions, formulated as an eigenvalue problem for a weighted Laplace-Beltrami operator. Ultimately, we provide some insight into the limit as $\gamma \to -2$, by showing the the convergence of this stability criterion to that associated with stable minimal cones.

\subsection{A stability criterion for cones} 
In the context of homogeneous solutions with isolated singularities, a stability condition can be reformulated as a lower bound for an eigenvalue problem associated with a Laplace-Beltrami operator on $\mathbb{S}^{d-1}$. As already observed in \cite{JerisonSavin15:NoSingularCones4D} for $\gamma = 0$ (i.e.,~the Alt-Caffarelli problem), and in \cite{Simons:minimal-varieties,CafHarSim} for the limiting case $\gamma \to -2$ (i.e.,~the minimal surfaces), this criterion serves as a key tool in ruling out the presence of singularities for homogeneous free boundaries (see also \cite{pacati} for the analogue in the capillary context).

The following result is the analogue of these criteria for the Alt-Phillips problem, where the eigenvalue problem involves a degenerate Laplace-Beltrami operator on $\mathbb{S}^{d-1}$. 
For clarity, it is stated in terms of the auxiliary function $w := \beta u^{1/\beta}$, although it can be equivalently formulated in terms of $u$. In the following, we denote by $\nabla_S$ and $ \mathrm{div}_S$ respectively the tangential gradient and the tangential divergence on  $\mathbb{S}^{d-1}$.
\begin{proposition}\label{p:stability-criterion-eigen}
    Let $\gamma \in (-2,0)$ and $w \in C^{0,1}(\R^d)$ be a $1$-homogeneous global regular solution outside the origin, in the sense of \cref{def:definition2.1}. Given the regular section $\Sigma_w:=\O_w \cap \mathbb{S}^{d-1}$, consider the degenerate eigenvalue problem on $\mathbb{S}^{d-1}$
    \be\label{e:lambda_s}
\lambda_s(\Sigma_w):= \min_{\substack{\varphi \in C^\infty(\mathbb{S}^{d-1})\\ \varphi\not\equiv0}}\ddfrac{\int_{\Sigma_w}  w^s|\nabla w|^2\big(|\nabla_S \varphi|^2 - \mathcal{A}_w^2 \varphi^2\big)\,d\HH^{d-1}}{\int_{\Sigma_w} w^s|\nabla w|^2 \varphi^2 \,d\HH^{d-1}}\,,
\ee
with $\mathcal{A}^2_w$ as in \eqref{e:stable2}. Then, $u$ is stable in $\R^d$ if and only if 
    \be\label{e:lower-bound-lambda}
    \lambda_s(\Sigma_w)\geq -\left(\frac{d+s-2}{2}\right)^2
    \ee
\end{proposition}
\begin{proof}
Through the proof, we often identity a point $x \in \R^d$ with its spherical coordinates $(r,\theta)$, with $r>0$ and $\theta \in \mathbb{S}^{d-1}$. Notice also that, since $w$ is a $1$-homogeneous function, we have that $w^s$ is $s$-homogeneous, $|\nabla w|$ is $0$-homogeneous and $\mathcal{A}_w^2$ is $(-2)$-homogeneous.\\
Thus, let $f(r,\theta):=g(r)\varphi(\theta)$ with $\varphi \in C^\infty(\mathbb{S}^{d-1})$ and $g \in C^\infty_c((0,+\infty))$. Using the coarea formula and the homogeneity of $w$, we get
\begin{align*}
\int_{\O_w} w^s |\nabla w|^2 \left(|\nabla f|^2 -\mathcal{A}^2_w f^2\right)\,dx&=\int_0^{\infty} r^{s+d-3}g^2(r)\,dr \int_{\Sigma_w} w^s|\nabla w|^2\left(|\nabla_S \varphi|^2 - \mathcal{A}_w^2\varphi^2\right) \,d\mathcal{H}^{d-1}\\
&\qquad + \int_0^{\infty} r^{s+d-1}(g'(r))^2\,dr \int_{\Sigma_w} w^s|\nabla w|^2 \varphi^2 \,d\mathcal{H}^{d-1}.
\end{align*}
On the other hand, since the Hardy's inequality implies that
$$
\inf\left\{\frac{\int_0^\infty r^{s+d-1}(g'(r))^2\,dr}{\int_0^\infty r^{s+d-3}g(r)^2\,dr} : g \in C^\infty_c((0,+\infty))\right\} = \left(\frac{d+s-2}{2}\right)^2,
$$
we deduce that, for $1$-homogeneous solutions, the stability condition \eqref{e:stable2} holds if and only if 
\be\label{e:real}
\int_{\Sigma_w} w^s|\nabla w|^2\left(|\nabla_S \varphi|^2 - \mathcal{A}_w^2\varphi^2\right) \,d\mathcal{H}^{d-1} + \left(\frac{d+s-2}{2}\right)^2 \int_{\Sigma_w} w^s|\nabla w|^2\varphi^2 \,d\mathcal{H}^{d-1}\geq 0
\ee
for every $\varphi \in C^\infty(\mathbb{S}^{d-1})$. Finally, by the definition of $\lambda_s(\Sigma_w)$, it is immediate to see that \eqref{e:real} is equivalent to the claimed lower bound \eqref{e:lower-bound-lambda}.
\end{proof}
In light of the stability condition for positive exponents \cite[Theorem 4.1]{thomas}, the previous result can be extended to the whole range $\gamma \in (-2,2)$. Naturally, in the case of non-negative exponents, the weighted eigenvalue $\lambda_s(\Sigma_w)$ admits alternative formulations depending on the notion of stability exploited in the proof (see the discussion in \cref{rem:differences}).\\\vspace{-0.35cm}

For the case of minimal surfaces, the stability criterion for cones can be stated as follows (see \cite{Simons:minimal-varieties}). Let $C$ be a minimal cone in $\R^d$, smooth outside the origin, so that $M:= C\cap \mathbb{S}^{d-1}$ is smooth in $\mathbb{S}^{d-1}$. Consider the eigenvalue problem 
    \be\label{e:eigen-ms}
\Lambda(M):= \min_{\substack{\phi \in C^\infty(\mathbb{S}^{d-1})\\ \phi\not\equiv0}}\ddfrac{\int_{M} \left(|\nabla_M \phi|^2 -|A_M|^2 \phi^2\right)\,d\HH^{d-2}}{\int_{M} \phi^2 \,d\HH^{d-2}}\,,
\ee
where $|A_M|^2$ is the squared norm of the second fundamental form of $M$ and $\nabla_M$ is the tangential gradient on $M$.
Then, $C$ is stable in $\R^d$ if and only if 
    \be\label{e:lower-bound-lambda-ms}
    \Lambda(M)\geq -\left(\frac{d-3}{2}\right)^2.
    \ee
The eigenvalue $\Lambda(M)$ is the first eigenvalue of Jacobi operator associated to $M$. Indeed, if we set $L_M:=\Delta_M + |A_M|^2$, where $\Delta_M$ is the Laplace-Beltrami operator on $M$, then there exists $\phi \colon \mathbb{S}^{d-1}\to \R$ satisfying
$$
-L_M \phi = \Lambda(M) \phi\quad \mbox{on }M,\qquad \int_M \phi^2 \,d\mathcal{H}^{d-2}=1.
$$
\subsection{Exploring the limit as $\gamma \to -2$}
In order to provide some insight into the limit of the stability criterion as $\gamma \to -2$, we begin by sketching the phenomena that arise in this singular regime. The key point in this discussion is the behavior of the weighted measure associated with \eqref{e:lambda_s} as $s \to -1$ (see \eqref{eq:citodopo} below).

Given $s_k \to -1$, let $w_k$ be a sequence of $1$-homogeneous global stable solutions, regular outside the origin, in the sense of \cref{def:definition3.2} and \cref{def:definition2.1}, such that
\be\label{e:melacito}
\text{$w_k\lvert_{\mathbb{S}^{d-1}} \in C^{2,\alpha}(\overline\Sigma_{w_k})$ and $\Sigma_{w_k}:=\Omega_{w_k}\cap\mathbb{S}^{d-1}$ is $C^{2,\alpha}$-regular, 
uniformly in $k$}. 
\ee
The results in \cref{s:smooth} do not provide estimates that are uniform in $k$. As far as we know,  uniform $C^{1,\alpha}$ estimates have been established only in \cite{DeSilvaSavinNegativePowerCompactness}, for minimizers of a normalized Alt-Phillips functional. The proof of assumption \eqref{e:melacito} is rather involved and would require a more detailed refinement of our regularity theory, which goes beyond the scope of this paper.

Under such compactness, we proceed by studying the asymptotic behavior of both the Euler-Lagrange equation and the Rayleigh-type quotient corresponding to \eqref{e:lambda_s}, as $k \to +\infty$. 
Roughly speaking, the proof consists
into showing that
\be\label{eq:citodopo}
(1+s_k)w_k^{s_k}|\nabla w_k|^2\,d\mathcal{H}^{d-1} \overset{\ast}{\rightharpoonup} d\mathcal{H}^{d-2} \res M,\quad\mbox{as }k\to +\infty,
\ee
where $M:=\lim_{k\to \infty}\partial \Sigma_{w_k}$. The presence of the coefficient $(1+s_k)$ is coherent with the normalization introduced in \cite[Theorem 2.4]{DeSilvaSavinNegativePower}, where they shown that the perimeter functional coincides with $\Gamma$-limit of $\mathcal{E}_s$, as $s\to -1$. Precisely, in the context of minimizers, they show that $
w_k^{1+s_k}\to \ind_{\{{w_\infty}>0\}}$ in $L^1$,
where $w_\infty:=\lim_{k\to+\infty}w_k$. 
Then \eqref{eq:citodopo} follows by noticing 
$$
|\nabla w_k^{1+s_k}| = (1+s_k)w_k^{s_k} |\nabla w_k|,\qquad\mbox{and}\qquad|\nabla w_k|=1\quad\text{on }\partial\Sigma_{w_k}.
$$ 

We also notice that \eqref{e:melacito} allows to deduce that $M$ is minimal, i.e.,~$H_{M}=0$. Indeed, let $H_k$ be the mean curvature of $\partial \Sigma_{w_k}$ oriented towards the complement $\O_{w_k}$ and $\nu_k$ be the outer normal to $\partial \Sigma_{w_k}$. Then
$$
w_k(x-t\nu_k(x)) = t  - \frac{H_k}{2(1+s_k)}t^2 + O(t^{2+\alpha}),\quad\mbox{for }x\in \partial \Sigma_{w_k}.
$$
In light of the higher regularity result of \cref{s:smooth}, the computation of the second normal derivative is essentially equivalent to one in \cite[Lemma 5.1]{thomas}. Hence, by substituting such expansion in the free boundary condition \eqref{e:hodo-pre-transform} (alternatively \eqref{e:visc.DeS.w}), we deduce that $\lim_{t \to 0^+} t^{1+s_k}H_k = 0$ for every $k>0$. Therefore, the claim follows by $C^{2,\alpha}$-compactness.\\\vspace{-0.35cm}
 
In the present section we proceed by underlying the effect of such phenomenon by showing the convergence of the stability criterion \cref{p:stability-criterion-eigen}, under the uniform regularity assumptions. \\ \vspace{-0.35cm}

\noindent \textit{Step 1: the asymptotic of the first eigenfunction.}
We begin by introducing some notations. Let $w:=w_k$, $s:=s_k$ and 
set 
$$
\Sigma:=\O_w \cap \mathbb{S}^{d-1},\qquad M_t:=\partial\{w>t\}\cap \mathbb{S}^{d-1}.
$$
Notice that $M_0=\partial \Sigma$. Since $w$ has an isolated singularity at the origin, both $\Sigma$ and $M_t$ are $C^{2,\alpha}$-regular in $\mathbb{S}^{d-1}$, uniformly in $k$. Now, let  $\varphi:=\varphi_{k}:\mathbb{S}^{d-1} \to \R$ be an eigenfunction corresponding to $\lambda_s:=\lambda_{s_k}(\Sigma)$. By rewriting the Euler-Lagrange equations associated to \eqref{e:lambda_s}, we get
\be\label{e:EL-eigen}
\begin{cases}
-\left(\mathrm{div}_S(w^s|\nabla w|^2 \nabla_S \varphi)  +\mathcal{A}_w^2 w^s|\nabla w|^2 \varphi\right) = \lambda_s  w^s|\nabla w|^2 \varphi &\quad\mbox{in }\Sigma\vspace{0.02cm},\\
w^s |\nabla w|^2 \nabla_S \varphi \cdot \nabla_S w = 0 &\quad\mbox{on }\partial \Sigma.
\end{cases}\ee
We stress that the boundary condition is understood in a limiting sense, as we approach $\partial \Sigma$.
Moreover, we assume that the following $L^2$-type normalization holds true
\be\label{e:normalization-lambdas}
(1+s)\int_{\mathbb{S}^{d-1}}w^s|\nabla w|^2 \varphi^2\,d\mathcal{H}^{d-1}=1,
\ee
for every $k>0$. In light of the uniform $C^{2,\alpha}$-regularity of $w$ and $\Sigma$, by \cite[Theorem 1.1]{TTV} we deduce that $\varphi\in C^{2,\alpha}(\overline{\Sigma})$, uniformly in $k$.\\
Now, for every $0\leq t \leq  \norm{w}{L^\infty(\mathbb{S}^{d-1})}$, we denote by $H_{M_t}$ the mean curvature of $M_t\subset \mathbb{S}^{d-1}$, pointing towards the complement of $\{w>t\}$, and with 
$$
\nu := -\frac{\nabla_S w}{|\nabla_S w|}
$$
the outer normal vector to $M_t$, tangent to $\mathbb{S}^{d-1}$. Thus, the decomposition
$$
\Delta_S \varphi = \Delta_{M_t} \varphi - H_{M_t} \varphi_{\nu} + \varphi_{\nu\nu}\quad \mbox{on } M_t,
$$
implies that 
\begin{align*}
\mathrm{div}_S(w^s|\nabla w|^2 \nabla_S \varphi)  + \mathcal{A}_w^2 w^s|\nabla w|^2 \varphi &= w^s|\nabla w|^2 (\Delta_{M_t} \varphi + \mathcal{A}_w^2 \varphi - H_{M_t}\varphi_{\nu})\\
&\qquad + w^s|\nabla w|^2\varphi_{\nu \nu} + \nabla_S (w^s|\nabla w|^2)\cdot \nabla_S \varphi\quad \mbox{on } M_t.
\end{align*}
On the other hand, since $\nu$ is orthogonal to $M_t$, we have
\begin{align*}
\partial_{\nu} (w^s|\nabla w|^2 \varphi_{\nu}) &= w^s |\nabla w|^2 \varphi_{{\nu} {\nu}} + \partial_{\nu}(w^s |\nabla w|^2)\varphi_{\nu}\\
&= w^s |\nabla w|^2 \varphi_{{\nu} {\nu}} + \nabla_S(w^s |\nabla w|^2)\cdot \nabla_S \varphi -\nabla_{M_t}(w^s |\nabla w|^2)\cdot \nabla_{M_t}\varphi\\
&= w^s |\nabla w|^2 \varphi_{{\nu} {\nu}} + \nabla_S(w^s |\nabla w|^2)\cdot \nabla_S \varphi - w^s\nabla_{M_t}(|\nabla w|^2)\cdot \nabla_{M_t}\varphi,
\end{align*}
where in the last equality we use that $w$ is constant on $M_t$. Finally, we infer that
\begin{align*}
&\mathrm{div}_S(w^s|\nabla w|^2 \nabla_S \varphi) +\mathcal{A}_w^2 w^s|\nabla w|^2 \varphi\\
&\quad\quad= w^s|\nabla w|^2 \left(\Delta_{M_t} \varphi + \mathcal{A}_w^2\varphi - H_{M_t}\varphi_{\nu} + \frac{\nabla_{M_t}(|\nabla w|^2)}{|\nabla w|^2}\cdot \nabla_{M_t}\varphi\right)  + \partial_{\nu} (w^s |\nabla w|^2\varphi_{\nu})
\end{align*}
on $M_t$. Hence, since $\varphi$ is a an eigenfunction associated to $\lambda_s$, by \eqref{e:EL-eigen} we get 
$$
- \partial_{\nu} (w^s |\nabla w|^2\varphi_{\nu}) = w^s |\nabla w|^2 F_t\qquad\mbox{on } M_t,
$$
where 
$$
F_t:= \Delta_{M_t} \varphi + \mathcal{A}^2_w \varphi +  \lambda_s \varphi - H_{M_t}\varphi_{\nu} +\frac{\nabla_{M_t}(|\nabla w|^2)}{|\nabla w|^2}\cdot \nabla_{M_t}\varphi.
$$
By combining \eqref{szformula} with the boundary condition from \eqref{e:EL-eigen} and the fact that $|\nabla w| = 1$ on $\partial\Sigma$, we have
$$
F_0 \equiv L_{\partial\Sigma} \varphi +\lambda_s \varphi \quad \mbox{on }\partial\Sigma,
$$ where $L_{\partial\Sigma}$ is the Jacobi operator associated to $\partial\Sigma$ defined above.

Let us denote by $N_\delta(\partial\Sigma)$ a $\delta$-neighborhood of $\partial\Sigma$ in $\Sigma$. Since $\partial\Sigma$ is $C^{2,\alpha}$-regular in $\mathbb{S}^{d-1}$, uniformly in $k$, there exists $\rho_0>0$ for which in $N_{\rho_0}(\partial\Sigma)$ we can exploit a Fermi-type coordinate system. Indeed, we set $x(\theta,\rho):=\chi_{z(\theta)}(\rho)$, where $z : U \to \mathbb{S}^{d-1}$ is a local parametrization of $\partial\Sigma$ and, fixed $z \in \partial\Sigma$, the map $\rho \mapsto \chi_{z}(\rho)$ is a solution of
$$
\begin{cases}
\chi_z'(\rho) = -\nu(\chi_z(\rho))&\quad\mbox{for }\rho \in (0,\rho_0) \\
\chi_z(0)=z.
\end{cases}
$$
For the sake of readability, we omit the dependence of the new coordinates from $\theta \in U$. Since $w \in C^{2,\alpha}(\overline{\Sigma})$ and $|\nabla w|=1$ on $\partial\Sigma$, then in 
$N_{\rho_0}(\partial\Sigma)$ we can rewrite the condition $w(x)=t$ as
$$
w(\chi_z(\tau))=t \quad\iff\quad w(z -\tau \nu + O(\tau)) = t \quad \iff\quad  t=\tau(1+O(\tau)).
$$
Therefore
\be\label{eq:richiamo0}
\begin{aligned}
(w^s|\nabla w|^2\nabla_S\varphi\cdot \nu)(\chi_{z}(\rho)) &= \int_0^\rho (\nabla_S(w^s|\nabla w|^2\nabla_S\varphi\cdot \nu)(\chi_z(\tau))\cdot \chi'_z(\tau)\,d\tau \\
&= \int_0^\rho (\nabla_S(w^s|\nabla w|^2\nabla_S\varphi\cdot \nu)(\chi_z(\tau))\cdot (-\nu(\chi_z(\tau)))\,d\tau \\
&= \int_0^\rho (w^s|\nabla w|^2 F_{w(\chi_z(\tau))})(\chi_z(\tau))\,d\tau \\&= \int_0^\rho (w^s|\nabla w|^2 F_{\tau(1+O(\tau))})(\chi_z(\tau))\,d\tau.
\end{aligned}
\ee
In the following computations we restrict to the values of $t$ for which $M_t\subset N_{\rho_0}(\partial\Sigma)$. Again, since $w \in C^{2,\alpha}(\overline{\Sigma})$, uniformly in $k$, there exists $C>0$ such that 
\be\label{eq:richiamo1}
|(w^s|\nabla w|^2)(\chi_z(\tau)) - \tau^s| \leq  C\tau^{s+1},\quad \mbox{in } N_{\rho_0}(\partial\Sigma).
\ee
Fixed $z\in \partial\Sigma$, if we set $f(\tau):=F_{\tau(1+O(\tau))}(\chi_z(\tau))$, we can exploit the one-dimensional nature of the construction, and get
\be\label{eq:richiamo2}
\begin{aligned}
\bigg|(1+s)\int_0^\rho (w^s|\nabla w|^2)&(\chi_z(\tau)) f(\tau)\,d\tau -
f(0)\bigg|
\\
&\leq (1+s)\int_0^\rho \tau^s |f(\tau)-f(0)|\,d\tau + C(1+s) \int_0^\rho \tau^{s+1}|f(\tau)|\,d\tau,\\
&\leq (1+s) C \left([f]_{C^{0,\alpha}}+\norm{f}{L^\infty} \right) \rho^{s+1+\alpha},
\end{aligned}
\ee
where the right-hand side converges to zero as $k\to +\infty$, being $f \in C^{0,\alpha}([0,\rho_0])$, uniformly in $k$. Notice also that $f(0)=F_0(z)$, for every $z \in \partial\Sigma$. Therefore, by exploiting the asymptotic of $w$ and $|\nabla w|$ close to $\partial\Sigma$, and combining \eqref{eq:richiamo0}, \eqref{eq:richiamo1}, \eqref{eq:richiamo2}, we get 
$$
|(1+s)\rho^s(\nabla_S\varphi\cdot \nu)(\chi_z(\rho)) - F_0(z)| \leq  (1+s)C \rho^{s+1},
$$
for every $\rho\in (0,\rho_0)$ and $z \in \partial\Sigma$. Similarly, by applying the coarea formula, the normalization \eqref{e:normalization-lambdas} implies 
\be\label{e:denominator}
\left|\int_{\partial\Sigma}\varphi^2\,d\mathcal{H}^{d-2}-1 \right|\leq (1+s)C\rho^{s+1}.
\ee
By rephrasing the last estimates with the original functions, we get 
$$
\left|L_{\partial \Sigma_{w_k}}\varphi_k + \lambda_{s_k}(\Sigma_{w_k})\varphi_k\right|
\leq (1+s_k)C\rho_0^{s+1}
\,\,\,\mbox{on }\partial \Sigma_{w_k},\quad\,\left|\int_{\partial \Sigma_{w_k}}\varphi_k^2\,d\mathcal{H}^{d-2}-1 \right|\leq (1+s_k)C\rho_0^{s+1}.
$$
Since $w_k,\varphi_k \in C^{2,\alpha}(\overline{\Sigma}_{w_k})$ and $\partial \Sigma_{w_k}$ is $C^{2,\alpha}$-regular, uniformly in $k$, we can proceed by passing to the limit as $k\to \infty$. If we set 
$$
w_\infty:=\lim_{k\to \infty}w_{k},\quad \varphi_\infty:=\lim_{k\to \infty}\varphi_{k},\quad M:=\lim_{k\to \infty}\partial \Sigma_{w_{k}},\quad \lambda_\infty:=\lim_{k\to+\infty}\lambda_{s_k}(\Sigma_{w_k})
$$
we get 
$$
-L_{M}\varphi_\infty =\lambda_{\infty}\varphi_\infty\quad \mbox{on }M,\qquad
\int_{M} \varphi_\infty\,d\mathcal{H}^{d-2}=1,\qquad 
\lambda_\infty \geq -\left(\frac{d-3}{2}\right)^2.
$$

\noindent\textit{Step 2: the asymptotic of the Rayleigh quotient.}
We want to conclude by showing that $\lambda_\infty=\Lambda(M)$, and so that $M$ is a stable minimal cone. A priori, the claim would follow by showing that the Rayleigh-type quotient in \eqref{e:lambda_s} converges to the one in \eqref{e:eigen-ms} as $k\to \infty$.  Nevertheless, as pointed out in \cref{rem:quadratic} below, such convergence does not hold in general. Our approach is therefore to prove this convergence for a particular family of test functions.\\
Thus, let us omit the dependence on the index $k$ and consider the same notations introduced above. Let $\varphi \in C^\infty(\mathbb{S}^{d-1})$ satisfying the normalization \eqref{e:normalization-lambdas}. We set 
\be\label{testfunctions}
\widetilde\varphi(\chi_{z(\theta)}(\tau))=\varphi(z(\theta))
\ee so that
$\widetilde\varphi\in C^\infty(\mathbb{S}^{d-1})$ and 
$$
\widetilde\varphi=\varphi\quad\text{on }\partial\Sigma,\qquad\nabla_S\widetilde\varphi\cdot \nu=0\quad\text{on }\partial\Sigma.
$$ 
Since $\widetilde\varphi$ is regular, we also have $w^s|\nabla w|^2 \nabla_S\widetilde\varphi \cdot \nu =0$ on $\partial\Sigma$.
In particular, since $\nabla_S = \nabla_{\partial\Sigma} + (\nabla_S \cdot \nu)\nu$, it implies that $\nabla_S\widetilde\varphi=\nabla_{\partial\Sigma}\varphi$. 
If we define $g(x):=|\nabla_S \widetilde\varphi|^2 - \mathcal{A}^2_w \widetilde\varphi^2$, then, by \eqref{szformula} and the definition on $\widetilde\varphi$, we have
\be\label{eq:tilde}
g\equiv |\nabla_{\partial\Sigma} \varphi|^2 - |A_{\partial\Sigma}|^2 \varphi^2  \quad \mbox{on }\partial\Sigma.
\ee
It is not restrictive to assume that $\mathrm{supp}\varphi \subset N_{\rho_0}(\partial\Sigma)$. In this neighborhood, the coordinate system $x(\theta,\tau):=\chi_{z(\theta)}(\tau)$ induce the following expansion of the volume element
$$
d\mathcal{H}^{n-1}(x)= (1+O(\tau))d\mathcal{H}^{d-2}(\theta)d\tau,
$$
where $d\mathcal{H}^{d-2}(\theta)$ is the $(d-2)$-dimensional surface element on $\partial\Sigma$ (see \cite[Remark 3]{CaffarelliJerisonKenig04:NoSingularCones3D}). Then,  
$$
\int_{\Sigma} w^s|\nabla w|^2 g\,d\mathcal{H}^{d-1} =  \int_{\partial\Sigma}\left(\int_0^{\rho_0}(w^s|\nabla w|^2 g)(\chi_{z(\theta)}(\tau))(1+O(\tau))\,d\tau\right)\,d\mathcal{H}^{d-2}(\theta).
$$
By combining that $\tau \mapsto g(\chi_{z(\theta)}(\tau)) \in C^{0,\alpha}([0,\rho])$ and the asymptotic in \eqref{eq:richiamo2}, we get
\be\label{eq:richiama100}
\begin{aligned}
\lim_{k\to+\infty}(1+s_k)\int_{\Sigma_{w_k}} w_k^{s_k}|\nabla w_k|^2 g\,d\mathcal{H}^{d-1} &=\int_{M}g(\chi_{z(\theta)}(0))\,d\mathcal{H}^{d-2} \\&=\int_{M}(|\nabla_{M}\varphi|^2-|A_{M}|^2 \varphi^2)\,d\mathcal{H}^{d-2},
\end{aligned}
\ee
where the last identity follows by \eqref{eq:tilde}. For what it concerns the asymptotic of the denominator in the Rayleigh quotient, by
combining $$\left|\int_{\partial \Sigma_{w_k}}\widetilde{\varphi}^2\,d\mathcal{H}^{d-2}-(1+s_k)\int_{\Sigma_{w_k}}w_k^{s_k}|\nabla w_k|^2 \widetilde\varphi^2\,d\mathcal{H}^{d-1}
 \right|\leq (1+s_k)C\rho_0^{s_k+1}$$ with \eqref{e:denominator}, we finally get that
$$
\lambda_\infty=\lim_{k\to+\infty}\lambda_{s_k}(\Sigma_{w_k})\le \ddfrac{\int_{M} \left(|\nabla_M \varphi|^2 -|A_M|^2 \varphi^2\right)\,d\HH^{d-2}}{\int_{M} \varphi^2 \,d\HH^{d-2}},\quad\mbox{for every }\varphi \in C^\infty(\mathbb{S}^{d-1}),
$$
which implies that $\lambda_\infty=\Lambda(M)$.

\begin{remark}[The quadratic form associated to the stability]\label{rem:quadratic} In the previous steps, we focused on the asymptotic behavior of homogeneous stable solutions of the Alt-Phillips problem as $\gamma \to -2$, motivated by their relevance in the classification of singular points. However, under the regularity assumption \eqref{e:melacito}, this asymptotic analysis can be readily extended to stable solutions regular outside the origin, in the sense of \cref{def:definition3.1} and \cref{def:definition2.1}, by showing that the limit of the associated free boundaries is indeed a stable minimal surface.

Our analysis highlights a deep difference between the asymptotic of stable regular solutions and the one of the quadratic form associated to the stability condition, i.e.,
$$
Q_s(f):=\int_{ \O_w} w^s|\nabla w|^2\left(|\nabla f|^2 - \mathcal{A}^2_w\, f^2\right)\,dx.
$$
An additional term appears in the limit of 
$Q_s$ as $s \to -1$, corresponding to the derivative of the test functions in the direction orthogonal to the limiting interface  $M := \lim_k \partial \Omega_{w_k}$.  Nevertheless, the stability of this limit interface can be obtained by choosing test functions of type \eqref{testfunctions}. This behavior is not unexpected, as it mirrors phenomena observed for the singular limit of stable solutions to the Allen-Cahn equation (see \cite{le1,Hiesmayr}). 
\end{remark}
\section*{Acknowledgements}
M.C. is supported by the European Research Council (ERC), under the European Union's Horizon 2020 research and innovation program, through the project ERC VAREG - {\em Variational approach to the regularity of the free boundaries} (grant agreement No. 853404). M.C. also acknowledges the MIUR Excellence Department Project awarded to the Department of Mathematics, University of Pisa, CUP I57G22000700001. M.C and G.T. are members of INdAM-GNAMPA. G.T. is supported by the GNAMPA project E5324001950001 - {\em PDE ellittiche che degenerano su varietà di dimensione bassa e frontiere libere molto sottili}.
\printbibliography
\end{document}